\documentclass[11pt]{amsart}
\usepackage{verbatim}
\usepackage{rotating} 
\usepackage{amssymb}
\usepackage[bookmarks,colorlinks,breaklinks]{hyperref}  
\hypersetup{linkcolor=blue,citecolor=blue,filecolor=blue,urlcolor=blue} 
\usepackage{mathrsfs,amsfonts,amsmath,amssymb,epsfig,amscd,xy,amsthm} 
 
\newtheorem{theorem}{Theorem}[section]
\newtheorem{proposition}[theorem]{Proposition}

\newtheorem{lemma}[theorem]{Lemma}

\newtheorem{corollary}[theorem]{Corollary}

\newtheorem{conjecture}[theorem]{Conjecture}

\theoremstyle{plain}
\numberwithin{equation}{subsection}

\theoremstyle{remark}
\newtheorem{claim}[theorem]{Claim}

\newcommand{\C}{{\mathbb C}}

\newcommand{\Q}{{\mathbb Q}}

\newcommand{\N}{{\mathbb N}}

\newcommand{\cC}{{\mathcal C}}

\newcommand{\tensor}{\otimes}

\newcommand{\Kbar}{\overline{K}}

\newcommand{\Qbar}{\bar{\Q}}

\DeclareMathOperator{\Gal}{Gal}

\newcommand{\bP}{{\mathbb P}}

\newcommand{\bG}{{\mathbb G}}

\newcommand{\bC}{{\mathbb C}}

\newcommand{\lra}{\longrightarrow}
\newcommand{\cO}{\mathcal{O}}

\newcommand{\cM}{\mathcal{M}}

\newcommand{\cL}{\mathcal{L}}
\newcommand\cLbar {\mathcal{\overline{L}}}

\newcommand\OK {\Omega_K}
\renewcommand\O{{\mathcal O}}
\renewcommand\P{{\mathbb P}}
\newcommand{\hhat}{{\widehat h}}

\author{D.~Ghioca}
\address{
Dragos Ghioca\\
Department of Mathematics\\
University of British Columbia\\
Vancouver, BC V6T 1Z2\\
Canada
}
\email{dghioca@math.ubc.ca}

\author{K.~D.~Nguyen}
\address{
Khoa D.~Nguyen \\
Department of Mathematics\\
University of British Columbia\\
And Pacific Institute for The Mathematical Sciences\\ 
Vancouver, BC V6T 1Z2, Canada}
\email{dknguyen@math.ubc.ca}

\author{H.~Ye}
\address{
Hexi Ye\\
Department of Mathematics\\
Zhejiang University\\
Hangzhou, 310027\\
China
}
\email{yehexi@gmail.com}


\begin{document}
	\title[Dynamical Manin-Mumford Conjecture]{The Dynamical Manin-Mumford Conjecture and the Dynamical Bogomolov Conjecture for endomorphisms of $(\bP^1)^n$}

	\begin{abstract}
	We prove Zhang's Dynamical Manin-Mumford Conjecture and Dynamical Bogomolov Conjecture for dominant endomorphisms $\Phi$ of $(\bP^1)^n$. We use the equidistribution theorem for points of small height with respect to an algebraic dynamical system, combined with an analysis of the symmetries of the Julia set for a rational function. 
\end{abstract} 
	  
	\maketitle


\section{Introduction}
\label{sec:introduction}

The Chebyshev polynomial of degree $d$ is the unique polynomial $T_d$ with the property that for each $z\in\C$, we have $T_d(z+1/z) = z^d + 1/z^d$. A  Latt\`es map $f:\bP^1\lra \bP^1$ is a rational function coming from the quotient of an affine map $L(z)=az+b$ on a torus $\mathcal{T}$ (elliptic curve), i.e. $f=\Theta\circ L\circ \Theta^{-1}$ with $\Theta: \mathcal{T}\to \P^1$ a finite-to-one holomorphic map; see \cite{lattes} by Milnor. For two rational functions $f$ and $g$, we say they are  (linearly) conjugate if there exists an automorphism $\eta$ of $\bP^1$ such that $f = \eta^{-1}\circ g\circ \eta$. We call \emph{exceptional} any rational map of degree $d>1$ which is conjugate either to $z^{\pm d}$, or to $\pm T_d(z)$, or to a Latt\`es map. 

As always in algebraic dynamics, given a self-map $f$ on a variety $X$, we denote by $f^n$ its $n$-th iterate (for any non-negative integer $n$, where $f^0$ denotes the identity map). We say that $x\in X$ is periodic if there exists $n\in\N$ such that $f^n(x)=x$; we call $x$ preperiodic if there exists $m\in\N$ such that $f^m(x)$ is periodic. Also, for a subvariety $V\subset X$, we say that $V$ is periodic if (the Zariski closure of) $f^n(V)$ equals $V$ for some $n\in\N$; similarly, we say that $V$ is preperiodic if (the Zariski closure of)  $f^m(V)$ is periodic.


\subsection{Statement of our main results}


We prove the following result.
\begin{theorem}
\label{main result}
Let $n$ be a positive integer, let $f_i\in \C(x)$ (for $i=1,\dots, n$) be  non-exceptional rational functions of degree $d_i\ge 2$, and let $V\subset (\bP^1)^n$ be an irreducible subvariety defined over $\C$. Assume: 
\begin{enumerate}
\item[(1)] either that $V$ contains a Zariski dense set of preperiodic points under the action of $(x_1,\dots, x_n)\mapsto (f_1(x_1),\dots, f_n(x_n))$ and that $d_1=d_2=\cdots =d_n$;  
\item[(2)] or that $f_1,\dots,f_n\in\Qbar(x)$, that $V$ is defined over $\Qbar$, and that there exists a Zariski dense sequence of points $(x_{1,i},\dots, x_{n,i})\in V(\Qbar)$ such that $\lim_{i\to\infty}\sum_{j=1}^n\hhat_{f_j}(x_{j,i})=0$, where $\hhat_{f_j}$ is the canonical height with respect to the rational function $f_j$. 
\end{enumerate}
Then there exists a finite set $S$ of tuples 
$$(i,j)\in\{1,\dots, n\}\times \{1,\dots, n\}$$ along with $(\ell_i,\ell_j)\in \N\times \N$ and  curves $C_{i,j}\subset \bP^1\times \bP^1$ which are preperiodic under the coordinatewise action $(x_i,x_j)\mapsto \left(f_i^{\ell_i}(x_i),f_j^{\ell_j}(x_j)\right)$ such that:
\begin{enumerate}
\item[(i)] $\deg\left(f_i^{\ell_i}\right)=\deg\left(f_j^{\ell_j}\right)$; and 
\item[(ii)] $V$ is an irreducible component of 
\begin{equation}
\label{boolean form}
\bigcap_{(i,j)\in S} \pi_{i,j}^{-1}\left(C_{i,j}\right),
\end{equation}
where $\pi_{i,j}:(\bP^1)^n\lra (\bP^1)^2$ is the projection on the $(i,j)$-th coordinate axes for each $(i,j)\in S$.
\end{enumerate} 
\end{theorem}

Our Theorem~\ref{main result} answers Zhang's Dynamical Manin-Mumford Conjecture (over $\C$) and a slightly more general form of Zhang's Dynamical Bogomolov Conjecture (over $\Qbar$) for endomorphisms of $(\bP^1)^n$ (see \cite[Conjectures~1.2.1~and~4.1.7]{ZhangLec}).  Note that  any dominant (regular) endomorphism of $(\bP^1)^n$ has an iterate which is of the form 
$$\Phi:=(f_1,\dots, f_n):(\bP^1)^n\lra (\bP^1)^n;$$
see also \cite[Remark~1.2]{GNY}. Our result is slightly stronger than the one conjectured in \cite{ZhangLec} since in part~(2) of Theorem~\ref{main result} we do not assume the endomorphism $\Phi=(f_1,\dots, f_n)$ is necessarily polarizable (i.e., the rational maps $f_i$ might have different degrees). On the other hand, we exclude the case when the $f_i$'s are conjugate to monomials, $\pm$Chebyshev polynomials, or Latt\`es maps since in those cases there are counterexamples to a formulation when $\Phi$ is not polarizable (see \cite{GTZ} and \cite[Remark~1.2]{GNY}). Moreover, if at least two of the maps $f_i$ are Latt\'es, then even assuming $\Phi$ is polarizable, one would still have to impose an additional condition in order to get that the subvariety $V$ is preperiodic (see \cite[Theorem~1.2]{GTZ}). In our next result (see Theorem~\ref{general DMM result}) we prove the appropriately modified statement of the Dynamical Manin-Mumford Conjecture (as formulated in \cite[Conjecture~2.4]{GTZ}) for all polarizable endomorphisms of $(\bP^1)^n$. 

\begin{theorem}
\label{general DMM result}
Let $n\in\N$, let $f_i\in\C(x)$ (for $i=1,\dots, n$) be rational functions of degree $d>1$, let $\Phi:(\bP^1)^n\lra (\bP^1)^n$ be defined by $$\Phi(x_1,\dots,x_n)=(f_1(x_1),\dots, f_n(x_n))$$ 
and let $V\subset (\bP^1)^n$ be an irreducible subvariety. Assume there exists a Zariski dense set of smooth points $P=(a_1,\dots, a_n)\in V(\C)$ which are preperiodic under $\Phi$ and moreover such that the tangent space $T_{V,P}$ of $V$ at $P$ is preperiodic under the induced action of $\Phi$ on  ${\rm Gr}_{\dim(V)}\left(T_{(\bP^1)^n,P}\right)$, where ${\rm Gr}_{\dim(V)}\left(T_{(\bP^1)^n,P}\right)$ is the corresponding Grassmannian. Then the subvariety $V$ must be preperiodic under the action of $\Phi$.
\end{theorem}


\subsection{Brief history of the Dynamical Manin-Mumford Conjecture and of the Dynamical Bogomolov Conjecture}


Motivated by the classical Manin-Mumford conjecture (proved by Laurent \cite{Laurent} in the case of tori, by Raynaud \cite{Raynaud} in the case of abelian varieties and by McQuillan \cite{McQuillan} in the general case of semiabelian varieties) and also by the classical Bogomolov conjecture (proved by Ullmo \cite{Ullmo} in the case of curves embedded in their Jacobian and by Zhang \cite{Zhang:Bogomolov} in the general case of abelian varieties), Zhang formulated dynamical analogues of both conjectures (see \cite[Conjecture 1.2.1, Conjecture 4.1.7]{ZhangLec}) for polarizable endomorphisms of any projective variety. We say that an endomorphism $\Phi$ of a projective variety $X$ is \emph{polarizable} if there exists an ample line bundle $\cL$ on $X$ such that $\Phi^*\cL$ is linearly equivalent to $\cL^{\otimes d}$ for some integer $d\ge 2$. As initially conjectured by Zhang, one might expect that if $X$ is defined over a field $K$ of characteristic $0$ and $\Phi$ is a polarizable endomorphism of $X$, and the  subvariety $V\subseteq X$ contains a Zariski dense set of preperiodic points, then $V$ is preperiodic.  Furthermore if $K$ is a number field then one can construct the canonical height $\hhat_\Phi$  for all points in $X(\Qbar)$ with respect to the action of $\Phi$ (see \cite{C-S} and also our  Subsection~\ref{height subsection}) and then Zhang's dynamical version of the Bogomolov Conjecture asks that if a subvariety $V\subseteq X$ is not preperiodic, then there exists $\epsilon>0$ with the property that the set of points $x\in V(\Qbar)$  such that $\hhat_\Phi(x)<\epsilon$ is not Zariski dense in $V$. Since all preperiodic points have canonical height equal to $0$, the Dynamical Bogomolov Conjecture is a generalization of the Dynamical Manin-Mumford Conjecture when the algebraic dynamical system $(X,\Phi)$ is defined over a number field.

Besides the case of abelian varieties $X$ endowed with the multiplication-by-$2$ map $\Phi$ (which motivated Zhang's conjectures), there are known only a handful of special cases of the Dynamical Manin-Mumford or the Dynamical Bogomolov conjectures. All of these partial results are for curves contained in $\bP^1\times \bP^1$---see \cite{Baker-Hsia, GT-Bogomolov, GTZ, GNY}. We also mention here the paper of Dujardin and Favre \cite{Favre-Dujardin} who prove a result for plane polynomial automorphisms motivated by Zhang's Dynamical Manin-Mumford Conjecture. Our Theorem~\ref{main result} is the first result towards the Dynamical Manin-Mumford and the Dynamical Bogomolov conjectures for higher dimensional subvarieties of $(\bP^1)^n$. 

The case $n=2$ in Theorems~\ref{main result} and \ref{general DMM result} (i.e., $V$ is a curve in $\bP^1\times \bP^1$) was established in \cite[Theorem~1.1~and~1.3]{GNY}. Even though the general strategy in our present proof follows the one we employed in \cite{GNY}, there are significant new obstacles that we need to overcome; for more details,  see Subsection~\ref{subsection proof strategy}.


\subsection{Preperiodic subvarieties}


The conclusion from Theorem~\ref{main result} covers the main result of Medvedev's PhD thesis \cite{Alice} (whose main findings were published in  \cite{Medvedev-Scanlon}) who showed that any invariant subvariety $V\subset (\bP^1)^n$ under the coordinatewise action of $n$ non-exceptional rational functions must have the form \eqref{boolean form}. Our result is stronger than the results from \cite{Alice, Medvedev-Scanlon} since we only assume that a subvariety $V\subset (\bP^1)^n$ contains a Zariski dense set of preperiodic points under the action of $\Phi:=(f_1,\dots, f_n)$ and then we derive that $V$ must have the form \eqref{boolean form} (see also our Theorem~\ref{shape preperiodic}). Medvedev and Scanlon assume that $V$ is invariant by $\Phi$ (or more generally, preperiodic under the action of $\Phi$) and then using the model theory of difference fields, they conclude that $V$ must have the form \eqref{boolean form}. We do not use model theory; instead, we use algebraic geometry (including the powerful Arithmetic Hodge Index Theorem of Yuan and Zhang \cite{Yuan-Zhang-published}) coupled with a careful analysis for the local symmetries of the Julia set of a rational function. We state below our formal result which covers the main result of \cite{Alice} thus providing the form of any preperiodic subvariety in $(\bP^1)^n$ under the split action of $n$ non-exceptional rational functions.

\begin{theorem}
\label{shape preperiodic}
Let $n\in\N$, let $f_1,\dots, f_n\in\C(x)$ be non-exceptional rational functions of degrees $>1$, and let $\Phi$ be their coordinatewise action on $(\bP^1)^n$. If $V\subset (\bP^1)^n$ is a preperiodic subvariety under the action of $\Phi$, then there exists a finite set $S$ of pairs $(i,j)\in\{1,\dots, n\}\times \{1,\dots, n\}$ along with curves $C_{i,j}\subset \bP^1\times \bP^1$ which are preperiodic under the coordinate wise action $(x_i,x_j)\mapsto (f_i(x_i),f_j(x_j))$ such that $V$ is an irreducible component of 
$\bigcap_{(i,j)\in S} \pi_{i,j}^{-1}\left(C_{i,j}\right)$, 
where $\pi_{i,j}:(\bP^1)^n\lra (\bP^1)^2$ is the projection on the $(i,j)$-th coordinate axes. 
\end{theorem}


\subsection{The Dynamical Pink-Zilber Conjecture}


Analogous to asking the Dynamical Manin-Mumford Conjecture as a dynamical variant of the classical Manin-Mumford conjecture, one could formulate a Dynamical Pink-Zilber Conjecture, at least in the case of split endomorphisms. The following statement is implicitly raised in \cite{DBHC}.

\begin{conjecture}
\label{DPZ conjecture}
Let $n\in\N$, let $f_1,\dots,f_n\in\C(x)$ be non-exceptional rational functions of degrees $>1$, and let $\Phi:(\bP^1)^n\lra (\bP^1)^n$ be their coordinatewise action $(x_1,\dots, x_n)\mapsto (f_1(x_1),\dots,f_n(x_n))$. For each $m\in\{0,\dots, n\}$, we let ${\rm Per}^{[m]}$ be the union of all  irreducible subvarieties of $(\bP^1)^n$ of codimension $m$, which are periodic  under the action of $\Phi$. If $V\subset (\bP^1)^n$ is an irreducible subvariety which is not contained in a proper periodic subvariety of $(\bP^1)^n$, then $V\cap {\rm Per}^{[\dim(V)+1]}$ is not Zariski dense in $V$.
\end{conjecture}

We exclude exceptional rational functions in Conjecture~\ref{DPZ conjecture} since in those cases we rediscover the classical Pink-Zilber Conjecture; for more details on the Pink-Zilber Conjecture, see \cite{Zannier}.

In Conjecture~\ref{DPZ conjecture}, if $V\subset (\bP^1)^n$ is an irreducible hypersurface, then we recover essentially the Dynamical Manin-Mumford Conjecture we proved in Theorem~\ref{main result}. Quite interestingly, the same Theorem~\ref{main result} can be used (along with other results) in order to solve Conjecture~\ref{DPZ conjecture} if $V\subset (\bP^1)^n$ has dimension $1$ or codimension $2$ and each $f_i$ is a polynomial defined over $\Qbar$; see \cite{DPZ}.


\subsection{Plan for our paper}


In Section~\ref{section reduction} we show that in order to prove Theorems~\ref{main result}~and~\ref{general DMM result}, it suffices to assume that $V\subset (\bP^1)^n$ is a hypersurface which projects dominantly onto any subset of $(n-1)$ coordinate axes of $(\bP^1)^n$. Thus we are left to prove our results for hypersurfaces $H$ (see Theorem~\ref{general hypersurface theorem}), which will be done over the remaining sections of our paper; the conclusion in Theorem~\ref{shape preperiodic} will follow from the ingredients we develop for proving Theorem~\ref{main result}. 

In Sections~\ref{section Julia set}~and~\ref{section equidistribution setup} we setup our notation, state basic properties for the Julia set of a rational function, construct the heights associated to an algebraic dynamical system and define adelic metrized line bundles which are employed in the main equidistribution result (Theorem~\ref{Yuan equidistribution}), which we will then use in our proof. In Section~\ref{section equal measures} we prove that under the hypotheses of Theorem~\ref{general hypersurface theorem}, the measures induced on the hypersuface $H\subset (\bP^1)^n$ from the dynamical systems $$\left((\bP^1)^{n-1}, f_1\times \cdots f_{i-1}\times f_{i+1}\times \cdots f_n\right)$$ 
are all equal (for $i=1,\dots, n$). Also, in Section~\ref{section equal measures} we prove Proposition~\ref{iff proposition}, which is a crucial step in our proof of our main results (for more details on this step and also on our overall proof strategy, see Subsection~\ref{subsection proof strategy}).

In Section~\ref{section from identical to preperiodic} we show how to use the equality of the above  measures to infer the preperiodicity of $H$, assuming also that $H$ satisfies an additional technical hypothesis (see Theorem~\ref{measure to periodicity}). In Section~\ref{section proof of main results} we finalize the proof of Theorem~\ref{general hypersurface theorem} (and thus finish our proof for Theorems~\ref{main result}~and~\ref{general DMM result}). We conclude our paper by proving Theorem~\ref{shape preperiodic}.  

\medskip  

{\bf Acknowledgments.} We are grateful to Tom Tucker, Xinyi Yuan and Shouwu Zhang for very useful conversations.  The first author is partially supported by an NSERC Discovery Grant; the second author is supported by a PIMS and a UBC postdoctoral fellowship. We also thank the Fields Institute for its hospitality and support during the last stage when this project was finalized.


\section{Reduction to the case of hypersurfaces}
\label{section reduction}


In this section we present various reductions which we will employ in  proving Theorems~\ref{main result}~and~\ref{general DMM result}. We also provide additional details regarding the overall strategy for our proof.


\subsection{Some reductions}


We start with the following important reduction.

\begin{proposition}
\label{prop first reduction}
It suffices to prove Theorems~\ref{main result}~and~\ref{general DMM result} under the additional hypothesis that $V\subset (\bP^1)^n$ is a hypersurface  which projects dominantly onto any subset of $n-1$ coordinate axes. 
\end{proposition}

\begin{proof}
First we prove that it suffices to assume in both Theorems~\ref{main result}~and~\ref{general DMM result} that $V\subset (\bP^1)^n$ is a hypersurface. Indeed, we assume Theorems~\ref{main result}~and~\ref{general DMM result} hold for all hypersurfaces and we derive the same conclusion for all subvarieties of $(\bP^1)^n$. So, let $V\subset (\bP^1)^n$ be an irreducible subvariety of dimension $D<n-1$ satisfying the hypotheses of Theorem~\ref{general DMM result} or hypothesis~(1) (or (2)) of Theorem~\ref{main result}. Then there exist $D$ coordinate axes (without loss of generality, we assume they are $x_1,\dots, x_D$) so that the projection $\pi$ of $(\bP^1)^n$ onto its first $D$ coordinate axes remains dominant when restricted to $V$. For each $j=D+1,\dots, n$, we let $\pi_j$ be the natural projection map of $(\bP^1)^n$ on coordinates $x_1,\dots, x_D,x_j$, and we let $H_j:=\pi_j(V)$. Then $H_j\subset (\bP^1)^{D+1}$ is a hypersurface satisfying the hypotheses of Theorem~\ref{general DMM result} or hypothesis~(1) (or (2)) of Theorem~\ref{main result} with respect to the coordinatewise action of the rational functions $f_1,\dots, f_D,f_j$. Furthermore, for each $j=D+1,\dots,n$, we let $\tilde{H}_j\subset (\bP^1)^n$ be the hypersurface $H_j\times (\bP^1)^{n-D-1}\subset (\bP^1)^n$ (i.e., we insert a copy of $\bP^1$ on each coordinate axis not included in the set $\{1,\dots, D,j\}$). Then also $\tilde{H}_j\subset (\bP^1)^n$ is a hypersurface satisfying the hypotheses of either Theorem~\ref{main result} or of Theorem~\ref{general DMM result}. Let 
\begin{equation}
\label{tilde H}
\tilde{H}:=\bigcap_{j=D+1}^n \tilde{H}_j;
\end{equation} 
clearly, $V\subset \tilde{H}$ and so, $D=\dim(V)\le \dim\left(\tilde{H}\right)$. 

Since $\dim(V)=D$ and $\pi|_V:V\lra (\bP^1)^D$ is a dominant morphism, then we conclude that there exists a Zariski open subset $U\subset (\bP^1)^D$ such that for each $\alpha\in U$, the fiber $\pi^{-1}(\alpha)$ is finite. Therefore for each $\alpha\in U$ and for each $j=D+1,\dots, n$, we have that there exists a finite set $S_{\alpha,j}$ with the property that if $(a_1,\dots, a_n)\in \tilde{H}_j$ and $(a_1,\dots, a_D)=\alpha$, then $a_j\in S_{\alpha,j}$. Thus for each $\alpha\in U$, we have that there exist finitely many points $(a_1,\dots, a_n)\in \tilde{H}$ such that $(a_1,\dots, a_D)=\alpha$. Hence $V$ is an irreducible component of $\tilde{H}$; moreover, any irreducible component $W$ of $\tilde{H}$ for which $\pi|_W:W\lra (\bP^1)^D$ is a dominant morphism has dimension $D$. 

If Theorem~\ref{main result} holds for hypersurfaces, then each hypersurface $\tilde{H}_j\subset (\bP^1)^n$ must have the form \eqref{boolean form} since each one of these hypersurfaces satisfies the hypotheses of Theorem~\ref{main result}. Actually, since each $\tilde{H}_j$ is a hypersurface, then we must have: 
$$\tilde{H}_j=\pi_{i,j}^{-1}(C_{i,j})$$ 
for some curve $C_{i,j}\subset \bP^1\times \bP^1$, which is preperiodic under the action of $(x_i,x_j)\mapsto \left(f_i^{\ell_i}(x_i),f_j^{\ell_j}(x_j)\right)$ for some $\ell_i,\ell_j\in\N$ with the property that $\deg\left(f_i^{\ell_i}\right)=\deg\left(f_j^{\ell_j}\right)$.  Because $V$ is an irreducible component of $\tilde{H}$ (see \eqref{tilde H}), we obtain the desired conclusion in Theorem~\ref{main result}.

Now, if Theorem~\ref{general DMM result} holds for hypersurfaces, then each hypersurface $\tilde{H}_j\subset (\bP^1)^n$ is preperiodic under the action of $\Phi:=(f_1,\dots, f_n)$. Thus, also $\tilde{H}$ is preperiodic under the action of $\Phi$ (see \eqref{tilde H}). Combining the following facts:
\begin{itemize}
\item $\tilde{H}$ is preperiodic; 
\item $V$ is an irreducible component of $\tilde{H}$;
\item each irreducible component $W$ of $\tilde{H}$ for which $\pi|_W:W\lra (\bP^1)^D$ is a dominant morphism has dimension $D$; and
\item each variety $\Phi^m(V)$ (for $m\in\N$) projects dominantly onto $(\bP^1)^D$, 
\end{itemize}
we obtain that $V$  itself must be preperiodic under the action of $\Phi$, as desired. 

Now, once we reduced proving Theorems~\ref{main result}~and~\ref{general DMM result} to the case $V\subset (\bP^1)^n$ is a hypersurface, we can reduce further to the special case when $V$ projects dominantly onto each subset of $(n-1)$ coordinate axes. Indeed, assuming otherwise, then (without loss of generality) we may assume $V=\bP^1\times V_0$ for some hypersurface $V_0\subset (\bP^1)^{n-1}$. Therefore, it suffices to prove Theorems~\ref{main result}~and~\ref{general DMM result} for the subvariety $V_0\subset (\bP^1)^{n-1}$ under the coordinatewise action of the rational functions $f_2,\dots,f_n$. A simple induction on $n$ finishes our proof.  (Finally, as a side note, we observe that in light of \cite[Theorem~1.1]{GNY}, then due to the reduction proved in Proposition~\ref{prop first reduction}, we have that  Theorem~\ref{main result} is equivalent with proving that if $n>2$ and also if each $f_i$ is non-exceptional, then there is no hypersurface $H\subset (\bP^1)^n$ projecting dominantly onto each subset of $(n-1)$ coordinate axes of $(\bP^1)^n$ such that $H$ contains a Zariski dense set of preperiodic points; this is exactly what we will be proving in Theorems~\ref{general hypersurface theorem}~and~\ref{measure to periodicity}.)   
\end{proof}


\subsection{A technical result}


The next result (proven in Section~\ref{section proof of main results}) in conjunction with \cite[Theorem~1.1~and~1.3]{GNY} yields the conclusions of both   Theorems~\ref{main result}~and~\ref{general DMM result}.

\begin{theorem}
\label{general hypersurface theorem}
Let $n>2$ be an integer, let $f_i\in \C(x)$ of degree $d_i\ge 2$ (for $i=1,\dots, n$) and let $H\subset (\bP^1)^n$ be an irreducible hypersurface projecting dominantly onto each subset of $(n-1)$ coordinate axes. If there is a Zariski dense sequence of points $(x_{1,i},\dots, x_{n,i})\in V(\C)$ such that:
\begin{enumerate}
\item[(1)] either each $(x_{1,i},\dots, x_{n,i})$ is preperiodic under the coordinatewise action of $(f_1,\dots, f_n)$ and also $d_1=d_2=\cdots =d_n$,  
\item[(2)] or each $f_i\in\Qbar(x)$ (for $i=1,\dots, n$), $V$ is defined over $\Qbar$ and $\lim_{i\to\infty}\sum_{j=1}^n\hhat_{f_j}(x_{j,i})=0$,
\end{enumerate}
then the following must hold:
\begin{enumerate}
\item[(i)] either each $f_i(x)$ is conjugate to $x^{\pm d_i}$ or to $\pm T_{d_i}(x)$,
\item[(ii)] or each $f_i$ is a Latt\'es map (for $i=1,\dots, n$).
\end{enumerate} 
\end{theorem}


\subsection{Our main results as consequences of the technical result}


We show next how to derive Theorems~\ref{main result}~and~\ref{general DMM result} from Theorem~\ref{general hypersurface theorem}.

\begin{proof}[Proof of Theorem~\ref{main result}.]
As shown in Proposition~\ref{prop first reduction}, it suffices to prove Theorem~\ref{main result} for irreducible hypersurfaces $V$, which project dominantly onto each subset of $(n-1)$ coordinate axes of $(\bP^1)^n$. Since no $f_i$ is exceptional, then Theorem~\ref{general hypersurface theorem} yields that the case of such hypersurfaces is vacuously true when $n>2$. The case of curves $V\subset (\bP^1)^2$ is proven in \cite[Theorem~1.1]{GNY}, which concludes our proof.
\end{proof}

\begin{proof}[Proof of Theorem~\ref{general DMM result}.]
Again using Proposition~\ref{prop first reduction}, it suffices to prove Theorem~\ref{main result} for irreducible hypersurfaces $V$, which project dominantly onto each subset of $(n-1)$ coordinate axes of $(\bP^1)^n$. The case $n=2$ was already proven in \cite[Theorem~1.3]{GNY}; so, from now on, we assume $n>2$. Then Theorem~\ref{general hypersurface theorem} yields that 
\begin{enumerate}
\item[(i)] either for each $i=1,\dots, n$ we have that  $f_i(x)=\nu_i^{-1}(x)\circ x^{\pm d}\circ \nu_i(x)$ or $f_i(x)=\nu_i^{-1}(x)\circ \left(\pm T_{d}(x)\right)\circ \nu_i(x)$ for some automorphisms $\nu_i:\bP^1\lra \bP^1$,
\item[(ii)] or each $f_i$ is a Latt\'es map corresponding to some elliptic curve $E_i$ (for $i=1,\dots, n$).
\end{enumerate} 
If condition~(i) is satisfied, then at the expense of replacing $V$ by $\tilde{\nu}(V)$, where $\tilde{\nu}$ is the automorphism of $(\bP^1)^n$ given by 
$$\tilde{\nu}(x_1,\dots, x_n):=(\nu_1(x_1),\dots, \nu_n(x_n)),$$  
we may assume that each $f_i(x)$ is either $x^{\pm d}$ or  $\pm T_{d}(x)$. Next, let $\mu:\bG_m^n\lra (\bP^1)^n$ be the morphism given by 
$$\mu(x_1,\dots,x_n)=:(\mu_1(x_1),\dots, \mu_n(x_n)),$$
for rational functions $\mu_i$ which are: 
\begin{itemize}
\item $\mu_i(x)=x$ if $f_i(x)=x^{\pm d}$; and 
\item $\mu_i(x)=x+\frac{1}{x}$ if $f_i(x)=\pm T_{d}(x).$
\end{itemize}
Then there exists an irreducible subvariety $W$ of $\mu^{-1}(V)\subset \bG_m^n$ (projecting dominantly onto $V$ through the map $\mu$), which contains a Zariski dense set of preperiodic points under the action of $\Phi:\bG_m^n\lra \bG_m^n$ given by 
$$(x_1,\dots, x_n)\mapsto \left(\pm x_1^{\pm d},\dots, \pm x_n^{\pm d}\right).$$ 
Hence, $W$ contains a Zariski dense set of torsion points of $\bG_m^n$. Laurent's theorem \cite{Laurent} (the original Manin-Mumford conjecture for tori) yields that $W$ is a subtorus, thus preperiodic under the action of $\Phi$. This proves that $V=\mu(W)$  is preperiodic under the action of 
$$(x_1,\dots, x_n)\mapsto (f_1(x_1),\dots, f_n(x_n)),$$
as desired.

Now, we assume condition~(ii) is verified and so, each $f_i$ is a Latt\'es map which satisfies $p_i\circ \psi_i=f_i\circ p_i$ where $p_i:E_i\lra \bP^1$ and $\psi_i:E_i\lra E_i$ are morphisms satisfying $\deg(\psi_1)=\deg(\psi_2)=\cdots = \deg(\psi_n)$ because the Latt\'es maps $f_i$ have the same degree. Then there exists an irreducible component $W$ of $p^{-1}(V)\subset \tilde{E}:=\prod_{i=1}^n E_i$ (where $p:\tilde{E}\lra (\bP^1)^n$ is the morphism given by $p_1\times \cdots \times p_n$) with the property that it contains a Zariski dense set of (smooth) points $P$ which are preperiodic under the action of the endomorphism $\tilde{\psi}$ of $\tilde{E}$ given by $\psi_1\times \cdots \psi_n$, and moreover, the tangent space of $W$ at $P$ is preperiodic under the induced action of $\tilde{\psi}$ on ${\rm Gr}_{\dim(W)}\left(T_{\tilde{E},P}\right)$, where $T_{\tilde{E},P}$ is the tangent space of $\tilde{E}$ at $P$ and ${\rm Gr}_{\dim(W)}\left(T_{\tilde{E},P}\right)$ is the corresponding Grassmannian. Since each $\psi_i$ is an isogeny of $E_i$ of same degree, we get that $\tilde{\psi}$ is a polarizable endomorphism of $\tilde{E}$ and so, \cite[Theorem~2.1]{GTZ} yields that $W$ is preperiodic under the action of $\tilde{\psi}$. Therefore $V=p(W)$ is preperiodic under the action of 
$$(x_1,\dots, x_n)\mapsto (f_1(x_1),\dots, f_n(x_n)),$$
as desired.
\end{proof}


\subsection{Strategy for our proof}
\label{subsection proof strategy}


The remaining sections of our paper are dedicated to proving Theorem~\ref{general hypersurface theorem}. The setup is as follows: 
\begin{itemize}
\item $n>2$ and  $H\subset (\bP^1)^n$ is a hypersurface projecting dominantly onto each subset of $(n-1)$ coordinate axes; 
\item $f_1,\dots, f_{n}$ are rational functions of degrees larger than $1$ acting coordinatewise on $(\bP^1)^n$; and 
\item $H$ contains a Zariski dense set either of preperiodic points or of points of small height (see~(2) in Theorem~\ref{main result}) under the action of $(x_1,\dots, x_n)\mapsto (f_1(x_1),\dots, f_n(x_n))$. 
\end{itemize}
If at least one of the functions $f_i$ is not exceptional, then we will derive a contradiction. Now, if some $f_i$ is conjugated to a monomial or  $\pm$Chebyshev polynomial, then we prove that each of the $n$ rational functions must be conjugated to a monomial or $\pm$Chebyshev polynomial. Similarly, if one of the $f_i$'s is a Latt\'es map, then we prove that each $f_i$ must be a Latt\'es map. We obtain this goal (see Theorem~\ref{measure to periodicity}) by showing a similitude between the Julia sets of each one of the rational functions $f_i$. In turn, the relation between the Julia sets is a consequence of a powerful equidistribution theorem for points of small height. 

More precisely, using the equidistribution theorem of \cite{Yuan} for points of small height on a variety (see \cite{C-L} for the case of curves and also \cite{BR} and \cite{favre-rivera06} for the case of $\bP^1$), we prove that under the above hypotheses for $H$ and the $f_i$'s, then the measures $\hat{\mu}_i$ induced on $H$ by the invariant measures corresponding to the dynamical systems 
$$\left((\bP^1)^{n-1}, f_1\times \cdots f_{i-1}\times f_{i+1}\times \cdots f_n\right)$$ 
are equal (for each $i=1,\dots, n$). Using a careful study of the local analytic maps which preserve (locally) the Julia set of a rational map (which is not exceptional), we obtain the conclusion of Theorem~\ref{main result}. Even though our arguments resemble the ones we employed in \cite{GNY} to treat the case of plane curves (i.e., $n=2$), there are significant new complications in our analysis. 

Indeed, using Yuan's arithmetic equidistribution theorem \cite{Yuan} for points with small height on a space of
dimension $n\geq 2$, we first get connections for the $(n-1, n-1)$-currents (coming from
dynamics) on a hypersurface $H\subset (\bP^1)^n$. From these connections, we are able
to construct many symmetries for the aforementioned $(n-1,n-1)$-current. A further analysis of the symmetries for such an 
$(n-1,n-1)$-current yields additional symmetries of the Julia set on the  $1$-dimensional slices of $(\bP^1)^n$. Applying the rigidity of the
symmetries of the Julia set on the $1$-dimensional slices, we are able to
derive the rigidity of the symmetries of the entire $(n-1,n-1)$-current, from
which we derive the desired conclusion regarding $H$ and the dynamical system $(f_1,\dots, f_n)$ (see the proof of Theorem~\ref{measure to periodicity}). It  is precisely the study of the rigidity of this $(n-1,n-1)$-current (for $n>2$) which  provides the new proof of Medvedev's result \cite{Alice}, which otherwise could 
not have been obtained from the arguments from our previous paper \cite{GNY}.

Also, in order to finish the proof of Theorem~\ref{general hypersurface theorem} by showing that the hypotheses of Theorem~\ref{measure to periodicity} are met,   we need to know that for a hypersurface $H\subset (\bP^1)^n$ as in Theorem~\ref{general hypersurface theorem}, 
for each point $(a_1, \dots, a_n)\in H$, 
\begin{equation}
\label{iff statement 1}
\text{if $a_1,\dots, a_{n-1}$ are preperiodic, then also $a_n$ is preperiodic.}
\end{equation} 
If $n=2$, this fact was known for quite some time (see \cite{Mimar} which publishes the findings of Mimar's PhD thesis \cite{Mimar-thesis} from 20 years ago). However, if $n>2$, in order to prove  \eqref{iff statement 1} (see our Proposition~\ref{iff proposition} in the case each $f_i$ and also $H$ are defined over $\Qbar$), we need to use both the classical Hodge Index Theorem (proved by Faltings \cite{Faltings-Hodge} and by Hriljac \cite{Hriljac} for arithmetic surfaces and proved by Moriwaki \cite{Moriwaki} for higher dimensional arithmetic varieties) and also the new Arithmetic Hodge Index Theorem (proved by Yuan and Zhang \cite{Yuan-Zhang-published}). Furthermore, in order to derive \eqref{iff statement 1} in the general case (over $\C$) we need a specialization argument based on a result of Yuan-Zhang \cite{Yuan-Zhang-1, Yuan-Zhang-2} regarding the specialization of a Zariski dense set of preperiodic points for a polarizable endomorphism defined over a base curve.


\section{Complex dynamics and height functions}
\label{section Julia set}
In this section, we introduce the Julia set of a rational function, some of its properties and also the arithmetic height functions associated to an algebraic  dynamical system. 


\subsection{The Julia set}  


Let $f:\bP^1\to \bP^1$ be a rational function defined over $\C$ of degree $d_f\geq 2$.  The {\em Julia set} $J_f$ is the set of points $x\in \bP^1_\C$ for which the dynamics is chaotic under the iteration of $f$.  The Julia set $J_f$  is closed, nonempty and invariant under $f$. Let $x$ be a periodic point in a cycle of exact period $n$; then the {\em multiplier} $\lambda$ of this cycle (or of the periodic point $x$) is the derivative of $f^n$ at $x$. A cycle is {\em repelling} if its multiplier has absolute value greater than $1$. All but finitely many cycles of $f$ are repelling, and repelling cycles are in the Julia set $J_f$. Locally, at a repelling fixed point $x$ with multiplier $\lambda$, we can conjugate $f$ to the linear map $z\to \lambda \cdot z$ near $z=0$ (note that $\lambda\ne 0$ since the point is assumed to be repelling). For more details about the dynamics of a rational function, we refer the reader to Milnor's book \cite{Milnor:book}. 

There is a probability measure $\mu_f$ on $\bP^1_\C$ associated to $f$, which is the unique $f$-invariant measure achieving maximal entropy $\log d_f$; see \cite{Bro, Lyu1, FLM, Man}.  Also $\mu_f$ is the unique measure satisfying 
\begin{equation}\label{measure growth}
\mu_f(f(A))=d_f\cdot \mu_f(A)
\end{equation} for any Borel set $A\subset \bP^1_\C$ with $f$ injective when restricted on $A$.  The support of $\mu_f$ is $J_f$, and $\mu_f(x)=0$ for any $x\in \bP^1_\C$. Moreover, $\mu_f$ has continuous potential, in the sense that locally there is a continuous subharmonic function $u(x)$ such that the $(1,1)$-current satisfies 
   $$dd^cu(x)=d\mu_f(x),$$
and then \eqref{measure growth} is equivalent to 
   $$dd^cu\circ f(x)=d_f\cdot d\mu_f(x).$$  


\subsection{Measures on a hypersurface associated to a dynamical system}
\label{subsection measures 2}


Let 
 $$\hat f(x_1, \cdots, x_{n}):=(f_1(x_1), \cdots, f_{n}(x_{n}))$$
be an endormorphism of $(\P^1_\C)^{n}$ with $f_i$ being a rational function of degree $d_i\geq 2$ for $1\leq i\leq n$. For $i=1, \cdots, n$, denote 
\begin{equation}
\label{tilde f i}
\tilde f_i:=(f_1, \cdots, f_{i-1}, f_{i+1}, \cdots, f_{n})
\end{equation}
as an endormorphism of $(\P_{\C}^1)^{n-1}$ with invariant measure 
 \begin{equation}\label{measure tilde i}\tilde \mu_i:=\mu_{f_1}\times \cdots \mu_{f_{i-1}}\times \mu_{f_{i+1}}\times \cdots \mu_{f_{n}}.
 \end{equation}
 Let $H\subset(\bP^1_\C)^{n}$ be an irreducible hypersurface projecting dominantly onto any subset of $(n-1)$ coordinates, i.e., the canonical projections $\hat{\pi}_i:(\bP^1)^{n}\to (\bP^1)^{n-1}$ (where for each  $i=1,\cdots, n$, $\hat \pi_i$ is the projection of $(\bP^1)^n$ onto the $(n-1)$ coordinates forgetting the $i$-th coordinate axis) restrict to dominant morphisms $\left(\hat{\pi}_i\right)|_H:H\lra (\bP^1)^{n-1}$.   By abuse of notation, we denote the restriction $(\hat \pi_i)|_H$ also by $\hat \pi_i$. We define probability measures $\hat{\mu}_{i}$ (for $i=1,\cdots, n$) on $H$ corresponding to the dynamical system $\left((\bP^1_{\C})^{n-1}, \tilde f_i\right)$. More precisely,  for each $i=1, \cdots, n$, we  pullback $\tilde \mu_i$ by $\hat \pi_i$ to get a measure $\hat \pi_i^*\mu_i$ on $H$ so that 
   $$\hat \pi_i^*\tilde \mu_i(A):=\tilde \mu_i(\hat \pi_i(A))$$
for any Borel set $A\subset H$ such that $\hat \pi_i$ is injective on $A$. Another way to interpret this is that for $t=(x_1, \cdots, x_{n})\in H$, we have that $d(\hat \pi_i^*\tilde \mu_i)(t)$ is an $(n-1,n-1)$-current on $H$ given by
   $$d\hat \pi_i^*\tilde \mu_i(t)=dd^cu_1(x_1)\wedge \cdot\cdot dd^cu_{i-1}(x_{i-1})\wedge dd^cu_{i+1}(x_{i+1})\cdot\cdot \wedge dd^cu_{n}(x_{n})$$
where $u_j$ is a locally defined continuous subharmonic function  with $dd^c u_j=d\mu_{f_j}$ for each $j=1,\dots, n$.     
    Hence we get the probability measures on $H$:
  $$\hat{\mu}_{i}:=\hat \pi_i^*\tilde \mu_i/\deg(\hat \pi_i)\text{ for $i=1, \cdots, n$.}$$
Similarly, one has that 
   $$\tilde f_i^*\tilde \mu_i=d_1\cdots d_{i-1}\cdot d_{i+1}\cdots d_{n}\cdot \tilde \mu_i \text{ for $i=1, \cdots, n$.}$$


\subsection{Symmetries of the Julia set}\label{sym}


Let $\zeta$ be a meromorphic function on some disc $B(a, r)$ of radius $r$ centred at a point $a\in J_f$. We say that $\zeta$ is a {\em symmetry} on $J_f$ if it satisfies the following properties:
\begin{itemize}
\item $x\in B(a, r)\cap J_f$ if and only if $\zeta(x)\in \zeta(B(a,r))\cap J_f$; and 
\item if $J_f$ is either a circle, a line segment, or the entire sphere, there is a constant $\alpha>0$ such that for any Borel set $A$ where $\zeta|_A$ is injective, one has $\mu_f(\zeta(A))=\alpha \cdot \mu_f(A)$. 
\end{itemize}
A family $\mathcal{S}$ of symmetries of $J_f$ on $B(a, r)$ is said to be {\em nontrivial} if $\mathcal{S}$ is normal on $B(a,r)$ and no infinite sequence $\{\zeta_n\}\subset \mathcal{S}$ converges to a constant function. A rational function is \emph{post-critically finite} (sometimes called critically finite), if each of its critical points has finite forward orbit, i.e. all critical points are preperiodic. According to Thurston \cite{Thu, DH}, there is an orbifold structure on $\bP^1$ corresponding to each post-critically finite map. A rational function is post-critically finite with {\em parabolic} orbifold if and only if it is exceptional; or equivalently its Julia set is smooth (a circle, a line segment or the entire sphere) with smooth maximal entropy measure on it; see \cite{DH}.


\subsection{The height functions}\label{height subsection} 


Let $K$ be a number field and $\Kbar$ be the algebraic closure of $K$. The number field $K$ is naturally equipped with a set $\OK$ of pairwise inequivalent nontrivial absolute values, together with positive integers $N_v$ for each $v\in \OK$ such that
\begin{itemize}
\item for each $\alpha \in K^*$, we have $|\alpha|_v=1$ for all but finitely many places $v\in \OK$. 
\item every $\alpha \in K^*$ satisfies the {\em product formula}
\begin{equation}\label{product formula}
   \prod_{v\in \OK} |\alpha|_v^{N_v}=1
   \end{equation}
\end{itemize}
For each $v\in \OK$, let $K_v$ be the completion of $K$ at $v$, let $\Kbar_v$ be the algebraic closure of $K_v$ and let $\C_v$ denote the completion of $\Kbar_v$. We fix an embedding of $\Kbar$ into $\C_v$ for each $v\in \OK$; hence we have a fixed extension of $|\cdot |_v$ on $\Kbar$. If $v$ is archimedean, then $\C_v\cong \C$. Let $f\in K(z)$ be a rational function with degree $d\geq 2$. There is a {\em canonical height} $\hhat_f$ on $\P^1(\Kbar)$ given by 
 \begin{equation}\label{C-S height}
\hhat_f(x):=\frac{1}{[K(x):K]} \lim_{n\to \infty}\sum_{T\in \Gal(\Kbar/K)\cdot X}~ \sum_{v\in \OK}N_v\cdot \frac{ \log\|F^n(T)\|_v}{d^n}
\end{equation}
 where $F:K^2\to K^2$ and $X$ are homogenous lifts of $f$ and respectively $x\in \P^1(\Kbar)$, while $\|(z_1, z_2)\|_v:=\max\{|z_1|_v, |z_2|_v\}$. By product formula (\ref{product formula}), the height $\hhat_f$ does not depend on the choice of the homogenous lift $F$ and therefore it is well-defined. As proven in   \cite{C-S}, $\hhat_f(x)\geq 0$ with equality if and only if $x$ is preperiodic under the iteration of $f$.


  \section{Adelic metrized line bundles and the equidistribution of points of small height}
\label{section equidistribution setup}


  In this section, we setup the  height functions and state the equidistribution theorem for points of small height, which would be used later in proving the main theorems of this article. The main tool we use here is the arithmetic equidistribution theorem for points with small height on algebraic varieties (see \cite{Yuan}).

  
\subsection{Adelic metrized line bundle} 
\label{subsection adelic metrized line bundles}


Let $\cL$ be an ample line bundle of an irreducible projective variety $V$ over a number field $K$. As in Subsection~\ref{height subsection}, $K$ is naturally equipped with absolute values $|\cdot|_v$ for $v\in \OK$. A {\em metric} $\|\cdot\|_v$ on $\cL$ is a collection of norms, one for each $t\in V(K_v)$, on the fibres  $\cL(t)$ of the line bundle, with 
   $$\|\alpha s(t)\|_v=|\alpha|_v\|s(t)\|_v$$
for any section $s$ of $\cL$. An {\em adelic (semipositive) metrized line bundle} $\cLbar=\{\cL, \{\|\cdot\|_v\}_{v\in \OK}\}$ over $\cL$ is a collection of metrics on $\cL$, one for each place $v\in \OK$, satisfying certain continuity and coherence conditions; see \cite{Zhang:line, Zhang:metrics, Yuan-Zhang-published}. 

There are various adelic metrized line bundles; the simplest adelic (semipositive) metrized line bundle is the line bundle $\O_{\P^1}(1)$ equipped with metrics $\|\cdot\|_v$ (for each $v\in \OK$), which evaluated at a section $s:=u_0Z_0 + u_1Z_1$ of $\O_{\bP^1}(1)$ (where $u_0, u_1$ are scalars and $Z_0, Z_1$ are the canonical sections of $\O_{\bP^1}(1)$) is given by 
   $$\|s([z_0: z_1])\|_v:=\frac{|u_0z_0+u_1z_1|_v}{\max\{|z_0|_v, |z_1|_v\}}.$$

Furthermore, we can define other metrics on $\O_{\P^1}(1)$ corresponding to a rational function $f$ of degree $d\geq 2$ defined over $K$. We fix  a homogenous lift  $F: K^2\to K^2$ of $f$ with homogenous degree $d$. For $j\geq 1$, write $F^j=(F_{0, j}, F_{1, j})$. For each place $v\in \OK$, we can define a metric on $\O_{\P^1}(1)$ as
\begin{equation}\label{metric of F}
\|s([z_0: z_1])\|_{v, F, j}:=\frac{|u_0z_0+u_1z_1|_v}{\max\{|F_{0, j}(z_0, z_1)|_v, |F_{1, j}(z_0, z_1)|_v\}^{1/d^j}},
\end{equation}
where $s=u_0Z_0+u_1Z_1$ with $u_0, u_1$ scalars and $Z_0, Z_1$ canonical sections of $\O_{\P^1}(1)$. Hence $\{\O_{\P^1}(1),\{ \|\cdot\|_{v, F, j}\}_{v\in \OK}\}$ is an adelic metrized line bundle over $\O_{\P^1}(1)$. 

A sequence $\{\cL, \{\|\cdot\|_{v,j}\}_{v\in \OK}\}_{j\geq 1}$ of adelic  metrized line bundles over an ample line bundle $\cL$ on a variety $V$ is convergent to $\{\cL, \{\|\cdot\|_v\}_{v\in \OK}\}$, if for all $j$ and all but finitely many $v\in\OK$, we have that  $\|\cdot\|_{v,j}=\|\cdot\|_v$, and moreover,  $\left\{\log\frac{\|\cdot\|_{v,j}}{\|\cdot\|_v}\right\}_{j\geq 1}$ converges to $0$ uniformly on $V(\C_v)$ for all $v\in \OK$. The limit $\{\cL, \{\|\cdot\|_v\}_{v\in \OK}\}$ is an adelic metrized line bundle. Also, the tensor product of two (adelic) metrized line bundles is again a (adelic) metrized line bundle. 

A typical example of a convergent sequence of adelic metrized line bundles is $\{\{\O_{\P^1}(1),\{ \|\cdot\|_{v, F, j}\}_{v\in \OK}\}\}_{j\geq 1}$
which converges to the metrized line bundle denoted by 
\begin{equation}\label{line F}
\overline{\cL}_F:=\{\O_{\P^1}(1), \{\|\cdot\|_{v,F}\}_{v\in\OK}\}
\end{equation}
    (see \cite{BR} and also see  \cite[Theorem~2.2]{Zhang:metrics} for the more general case of a polarizable endomorphism $f$ of a projective variety). 

As usual, we let $\tilde f=(f_1, \cdots, f_n)$ with $f_i$ being a rational function of degree $d_i\ge 2$ defined over the number field $K$ for $1\leq i\leq n$. Fix a homogenous lift $F_i$ for each $f_i$ and denote 
   $$\tilde F:=(F_1, \cdots, F_n).$$  
We let $\pi_i$ be the $i$-th coordination projection map from $(\P^1)^n$ to $\P^1$. We construct an adelic metrized line bundle on $(\P^1)^n$ as follows
\begin{equation}\label{mlb F}
\overline \cL_{\tilde F}:=\{\cL_{\tilde F}, \|\cdot \|_{v, {\tilde F}}\}:=(\pi_1^*\overline{\mathcal{L}}_{F_1})\tensor (\pi_2^*\overline{\mathcal{L}}_{F_2})\cdots \tensor (\pi_n^*\overline{\mathcal{L}}_{F_n}).
\end{equation}
where the metric $\|\cdot \|_{v, \tilde F}$ on $ \cL_{\tilde F}$ is the one inherited from the metrics $\|\cdot\|_{v, F_i}$ on $\O_{\P^1}(1)$ for $1\leq i\leq n$.


\subsection{Equidistribution of small points}
\label{subsection equidistribution small points}


For a semipositive metrized line bundle $\cLbar$ on a (irreducible and projective) variety $V$ defined over a number field $K$, the height for $t\in V(\Kbar)$ is given by 
\begin{equation}\label{points height}
\hhat_{\cLbar}(t)=\frac{1}{|\Gal(\Kbar/K)\cdot t|}\sum_{y\in\Gal(\Kbar/K)\cdot t}~\sum_{v\in \OK}-N_v\cdot \log\|s(y)\|_v
\end{equation}
where $|\Gal(\Kbar/K)\cdot t|$ is the number of points in the Galois orbits of $t$, and $s$ is any meromorphic section of $\cL$ with support disjoint from $\Gal(\Kbar/K)\cdot t$. A sequence of points $t_j\in V(\Kbar)$ is {\em small}  if  $\lim_{j\to \infty} \hhat_{\cLbar}(t_j)=\hhat_{\cLbar}(V)$, and is {\em generic} if no subsequence of $t_j$ is contained in a proper Zariski closed subset of $V$; see \cite{Zhang:metrics} for more details on constructing the height for any irreducible subvariety $Y$ of $V$ (which is denoted by $\hhat_\cLbar(Y)$). We use the following equidistribution result due to Yuan \cite{Yuan}  in the case of an arbitrary projective variety. 

\begin{theorem}\cite[Theorem 3.1]{Yuan}\label{Yuan equidistribution}
Let $V$ be a projective irreducible variety of dimension $n$ defined over a number field $K$, and let $\cLbar$ be a metrized line bundle over $V$ such that $\cL$ is ample and the metric is semipositive. Let $\{t_n\}$ be a generic sequence of points in $V(\Kbar)$ which is small. Then for any $v\in \OK$, the Galois orbits of the sequence $\{t_j\}$ are equidistributed in the analytic space $V^{an}_{\C_v}$ with respect to the probability measure $d\mu_v=c_1(\cLbar)^n_v/\deg_\cL(V)$. 
\end{theorem}

When $v$ is archimedean, $V_{\C_v}^{an}$ corresponds to $V(\C)$ and the curvature $c_1(\cLbar)_v$ of the metric $\|\cdot\|_v$ is given by $c_1(\cLbar)_v=\frac{\partial \overline{\partial}}{\pi i}\log \|\cdot\|_v$. If $v$ is a  non-archimedean place, then $V_{\C_v}^{an}$ is the Berkovich space associated to $V(\C_v)$, and Chambert-Loir \cite{C-L} constructed an analog for the  curvature on $V_{\C_v}^{an}$. The precise meaning of the equidistribution statement in Theorem~\ref{Yuan equidistribution} is that  
\begin{equation}
\label{limit statement 1}
\lim_{j\to \infty} \frac{1}{|\Gal(\Kbar/K)\cdot t_j|}\sum_{y\in \Gal(\Kbar/K)\cdot t_j} \delta_{y}=\mu_v,
\end{equation}
where $\delta_{y}$ is the point mass probability measure supported on $y\in V^{an}_{\C_v}$, while the limit from \eqref{limit statement 1} is the weak limit for the corresponding probability measures on the compact space $V^{an}_{\C_v}$. 


\subsection{Some examples} 


For the dynamical system $(\bP^1, f)$ corresponding to a rational function $f$ defined over a number field $K$ and of degree $d_f\geq 2$, at an archimedean place $v$, it is well known that the curvature of the limit of the metrized line bundles 
  $$\{\O_{\P^1}(1),\{ \|\cdot\|_{v, F, j}\}_{v\in \OK}\}_{j\geq 1}$$
  is a $(1,1)$-current given by $d\mu_f$, which is independent on the choice of $F$. Combining the definition \eqref{C-S height} of the canonical height $\hhat_f$ of $f$,  with the height \eqref{points height}  of points for an adelic metrized line bundle and the definition (\ref{metric of F}, \ref{line F}) of  $\overline \cL_F$, we get
    $$\hhat_{{\overline \cL}_F}(x)=\hhat_f(x)$$
which is independent of the choice for the lift $F$ of $f$. 

We conclude this section by noting that in the case of the metrized line bundle $\cLbar_{\tilde F}$ on $(\P^1)^n$ associated to an endomorphism $\tilde f$ of $(\P^1)^n$ (see subsection \ref{subsection adelic metrized line bundles}),  at an archimedean place $v$, the $(n, n)$-current satisfies the formula:
\begin{equation}
\label{equation c_1}
c_1(\cLbar_{\tilde F})^n_v=n!\cdot {d\tilde \mu}, 
\end{equation}
where $\tilde \mu=\mu_{f_1}\times \cdots \times \mu_{f_n}$ is the invariant measure on $(\bP^1_{\C_v})^n$ associated to the endomorphism $\tilde f=(f_1, \cdots, f_n)$. To see this, we first notice that since $v$ is archimedean, then $\C_v=\C$ and so, by taking $\frac{\partial \overline{\partial}}{\pi i}\log \|\cdot\|_{v, \tilde F}$ we get  
\begin{equation}\label{chen class 1}
c_1(\cLbar_{\tilde F})_v=dd^c (u_1(x_1)+\cdots + u_n(x_n)),
\end{equation}
where $u_i(x_i)$ is a locally defined continuous subharmonic function on $\P_{\C_v}^1$ with $dd^c u_i=d\mu_{f_i}$ for $1\leq i\leq n$. Hence 
  $$c_1(\cLbar_{\tilde F})_v^n=n!\cdot dd^c u_1(x_1)\wedge \cdots \wedge dd^c u_n(x_n)=n!\cdot d\tilde \mu,$$
and so, the equality from \eqref{equation c_1} follows. Moreover, for a point $t=(a_1, \cdots, a_n)\in (\P^1)^n(\Kbar)$, from \eqref{mlb F} we see that
\begin{equation}
\label{canonical height metrized statement}
\hhat_{\overline \cL_{\tilde F}}(t)=\hhat_{f_1}(a_1)+\cdots+\hhat_{f_n}(a_n).
\end{equation}


\section{Measures and heights on a hypersurface}
\label{section equal measures}
 
In this section we study the measures and the corresponding heights on a hypersurface in $(\bP^1)^n$; this allows us to obtain two important technical ingredients (Theorem~\ref{equal measure thm} and Proposition~\ref{iff proposition}) which will later be used in proving Theorem~\ref{general hypersurface theorem}. So, let $\hat f=(f_1, \cdots, f_{n})$ be an endomorphism of $(\bP^1)^n$ defined over a number field $K$, with degrees $d_i\geq 2$ for each rational function $f_i$ (for $1\leq i\leq n$). Also, let $H\subset (\bP^1)^{n}$ be an irreducible hypersurface defined over $K$, which projects dominantly onto each subset of $(n-1)$ coordinate axes.


\subsection{Adelic metrized line bundles on the hypersurface}
\label{subsection adelic metrics hypersurface}


For each $i=1,\dots, n$, as in \eqref{tilde f i}, we let $\tilde f_i$ be the endomorphism of $(\bP^1)^{n-1}$ given by forgetting the $i$-th coordinate axis (along with the action of $f_i$) in the dynamical system $\left((\bP^1)^n,\hat f\right)$. 
Let $\tilde F_i$ be a homogenous lift of $\tilde f_i$ as in Subsection~\ref{subsection measures 2} and then similar to \eqref{mlb F}, we construct an  adelic metrized line bundle $\overline \cL_{\tilde F_i}$ on $(\P^1)^{n-1}$ such that when $v$ is archimedean, we have  
  $$c_1(\overline \cL_{\tilde F_i})^{n-1}_v=(n-1)!\cdot d\tilde \mu_i$$
(for each $1\leq i\leq n$), where the probability measure $\tilde \mu_i$ on $(\P_{\C_v}^1)^{n-1}$ is the one appearing in \eqref{measure tilde i}. 

For each $i=1,\dots, n$, we recall from Subsection \ref{subsection measures 2} that the projection 
   $$\hat \pi_i:H\lra (\bP^1)^{n-1}$$
is the one given by forgetting the $i$-th coordinates; $\hat \pi_i$ is a finite, dominant morphism (due to our assumption on $H$). We let 
\begin{equation}
\label{cL hat F i}
\overline \cL_{{\hat F}_i}:=\hat \pi_i^*\overline \cL_{{\tilde F}_i}
\end{equation}
be an adelic metrized line bundle on $H$, which is the pullback of the adelic metrized line bundle $\overline \cL_{{\tilde F}_i}$ (on $(\P^1)^{n-1}$) by the morphism $\hat \pi_i$. 


\subsection{Height functions on the hypersurface} 


For each $i=1, \cdots, n$ and each $t=(a_1, \cdots, a_{n})\in H(\Kbar)$, using \eqref{canonical height metrized statement} we conclude that 
\begin{equation}
\label{equation height cL i}
\hhat_{\overline \cL_{{\hat F}_i}}(t)=\hhat_{f_1}(a_1)+\cdots+\hhat_{f_{i-1}}(a_{i-1})+\hhat_{f_{i+1}}(a_{i+1})+\cdots \hhat_{f_{n}}(a_{n}).
\end{equation}
Hence $\hhat_{\overline \cL_{{\hat F}_i}}(t)\geq 0$ with equality if and only if $a_j$ is preperiodic under $f_j$ for each $j\ne i$ with $1\leq j\leq n$. So, if the set of all $t\in H(\Kbar)$ for which $\hhat_{\overline \cL_{{\hat F}_i}}(t)=0$ is Zariski-dense on $H$, then each essential minima $e_j\left(\overline \cL_{{\hat F}_i}\right)$ (for $j=1,\dots, n$, defined as in \cite{Zhang:metrics}) are equal to $0$. Therefore, using the inequality from  \cite[Theorem 1.10]{Zhang:metrics}, we conclude that 
\begin{equation}\label{height of H}
\hhat_{\overline \cL_{{\hat F}_i}}(H)=0.
\end{equation}


\subsection{Equal measures on the hypersurface} 


Now we are ready to prove the following result. 
\begin{theorem}\label{equal measure thm}
Suppose that there is a generic sequence of points $t_j=(x_{1,j}, \cdots, x_{n, j})\in H(\Kbar)$ such that 
   $$\lim_{j\to\infty} \hhat_{f_1}(x_{1,j}) +\cdots+ \hhat_{f_{n}}(x_{n,j})=0.$$
 Then $\hat \mu_1=\hat \mu_2=\cdots =\hat \mu_{n}$.
\end{theorem}

\begin{proof}
This is an immediate consequence of Theorem~\ref{Yuan equidistribution} applied to the sequence of points $t_j=(x_{1,j}, \cdots, x_{n,j})\in H(\Kbar)$ with respect to the adelic metrized line bundles $\overline \cL_{\hat F_i}$ for $1\leq i\leq n$. Indeed, when $v$ is archimedean, using \eqref{height of H} and the assumption on the points $t_j\in H$ we get that the Galois orbits of $t_j$ in $H$ equidistribute with respect to the probability measures $\hat \mu_i$ on $H(\C)$ for each $i\in\{1,\dots, n\}$. Hence $\hat \mu_1=\hat \mu_2=\cdots =\hat \mu_n$. 
\end{proof}


\subsection{Preperiodic points on hypersurfaces}
\label{subsection preperiodic iff}


In this section we prove the following important result; we thank Xinyi Yuan and Shouwu Zhang for several very helpful conversations regarding its proof.
\begin{proposition}
\label{iff proposition}
Let $n\ge 2$ be an integer, let $f_1,\dots,f_n\in\C(x)$ be rational functions of degrees $d_i\ge 2$ and let $H\subset (\bP^1)^n$ be an irreducible hypersurface which projects dominantly onto any subset of $(n-1)$ coordinate axes. Assume:
\begin{enumerate}
\item[(1)] either that $H$ contains a Zariski dense set of preperiodic points under the action of $(x_1,\dots, x_n)\mapsto (f_1(x_1),\dots, f_n(x_n))$ and that $d_1=\cdots =d_n$;  
\item[(2)] or that $f_1,\dots,f_n\in\Qbar(x)$, that $H$ is defined over $\Qbar$, and that there exists a Zariski dense sequence of points $(x_{1,j},\dots, x_{n,j})\in H(\Qbar)$ such that $\lim_{j\to\infty}\sum_{i=1}^n\hhat_{f_i}(x_{i,j})=0$, where $\hhat_{f_i}$ is the canonical height with respect to the rational function $f_i$. 
\end{enumerate}
Then there exists $i\in\{1,\dots, n\}$ such that for any $(a_1,\dots, a_n)\in H(\C)$ for which $a_j$ is preperiodic under the action of $f_j$ for each $j\in\{1,\dots, n\}\setminus\{i\}$, then we have that also $a_i$ is preperiodic under the action of $f_i$.
\end{proposition}

We prove first that hypothesis~(2) in Proposition~\ref{iff proposition} yields the desired conclusion, and then we prove that part~(1) may be reduced to part~(2) in Proposition~\ref{iff proposition} through a specialization result of Yuan-Zhang \cite{Yuan-Zhang-1, Yuan-Zhang-2}. 

\begin{proof}[Proof of Proposition~\ref{iff proposition} assuming hypothesis~(2) holds.]
Since the case $n=2$ was proven in \cite{Mimar-thesis} (see also \cite{Mimar}), from now on, we assume $n>2$. We assume each $f_i\in\Qbar(x)$ and also that $H$ is defined over $\Qbar$. 

We use the notation as in Subsection~\ref{subsection adelic metrics hypersurface}; so, we consider the adelic metrics $\cLbar_{\hat F_i}$ (for $i=1,\dots, n$) on $H$, defined as in \eqref{cL hat F i}. For the sake of simplifying our notation, we will denote from now on the tensor product of two line bundles $\cM_1$ and $\cM_2$ as $\cM_1+\cM_2$. We denote by $\overline{{\rm Pic}}(H)$ the group of (adelic) metrized line bundles on $H$.   

\begin{lemma}
\label{lemma c i}
There exist real numbers $c_i$ (for $i=1,\dots, n$) not all equal to $0$ such that the metrized line bundle
\begin{equation}
\label{c i 0}
\cLbar_0:=c_1\cdot \cLbar_{\hat F_1} + \cdots + c_n \cLbar_{\hat F_n}\in \overline{{\rm Pic}}(H)\otimes \mathbb{R}
\end{equation}  
has the property that $\cLbar_0\cdot x=\hhat_{\cLbar_0}(x)=0$ for each $x\in H(\Qbar)$.  
\end{lemma}

\begin{proof}[Proof of Lemma~\ref{lemma c i}.]
We thank Shouwu Zhang for suggesting us the proof of this Lemma, which follows the idea used in the proof of \cite[Theorem~4.13]{Yuan-Zhang-published}. 

We let $\hat \cL_i\in {\rm Pic}(H)$ be the line bundle supporting $\cLbar_{\hat F_i}$, i.e., 
$$\hat \cL_i:=\pi_1^*\cO_{\P^1}(1)\otimes \cdots \otimes \pi_{i-1}^*\cO_{\P^1}(1)\otimes \pi_{i+1}^*\cO_{\P^1}(1)\otimes \cdots \otimes \pi_n^*\cO_{\P^1}(1),$$
where $\pi_j$ is the induced projection map of $H$ onto the $j$-th coordinate axis of $(\bP^1)^n$ (for each $j=1,\dots, n$). 

\begin{claim}
\label{claim c i}
There exist real constants $c_1,\dots, c_n$ (not all equal to $0$) such that the line bundle $\cL_0:=\sum_{i=1}^n c_i \hat \cL_i\in {\rm Pic}(H)\otimes \mathbb{R}$ is numerically trivial.
\end{claim}

\begin{proof}[Proof of Claim~\ref{claim c i}.]
The main ingredient in our proof is the classical Hodge Index Theorem (see \cite[Theorem~5.20]{Yuan-Zhang-published}). We let 
\begin{equation}
\label{c L 1 definition}
\cL_1:=\sum_{i=1}^n \hat \cL_i \in {\rm Pic}(H);
\end{equation}
then $\cL_1$ is ample (since it is the pullback of an ample line bundle on $(\bP^1)^n$ under the natural inclusion map $H\hookrightarrow (\bP^1)^n$).  We find the real numbers $c_i$ so that $\cL_0:=\sum_{i=1}^n c_i\hat \cL_i$ satisfies the following two conditions:
\begin{enumerate}
\item[(A)] $\cL_0\cdot \cL_1^{n-2} =0$ and
\item[(B)] $\cL_0^2\cdot \cL_1^{n-3} = 0$.
\end{enumerate}
Condition (A) above yields a linear relation between the unknowns $c_i$. On the other hand, condition~(B) yields a quadratic form in the variables $c_i$. This quadratic form is not positive-definite since (from the Hodge Index Theorem) we know that generically, for any line bundle $\cM$ satisfying $\cM\cdot \cL_1^{n-2}=0$, we have that $\cM^2\cdot \cL_1^{n-3}\le 0$. Also, this quadratic form is not negative definite since $\cL_1^2\cdot \cL_1^{n-3}=\cL_1^{n-1} > 0$ (because $\cL_1$ is ample). Therefore, there exist real numbers $c_i$, not all equal to $0$ such that $\cL_0$ satisfies both conditions~(A)~and~(B) above. Then \cite[Theorem~5.20]{Yuan-Zhang-published} yields that $\cL_0$ is numerically trivial, as claimed.  
\end{proof}

Let now $c_1,\dots, c_n\in\mathbb{R}$ satisfy the conclusion of Claim~\ref{claim c i} and define 
$$\cLbar_0:=\sum_{i=1}^n c_i\cLbar_{\hat F_i}\in \overline{\rm Pic}(H)\otimes \mathbb{R}.$$

We consider next the adelic metrized line bundle $\cLbar_1:=\sum_{i=1}^n\overline \cL_{\hat F_i}$; note that the generic fiber of $\cLbar_1$ is the ample line bundle $\cL_1$ from \eqref{c L 1 definition}. Using our hypothesis~(2) from Proposition~\ref{iff proposition}, i.e., the existence of a Zariski dense set of points on $H$ of height tending to $0$, we obtain that each of the successive minima $e_j(\cLbar_1)=0$ for $j=0,\dots, n-1$. Note that for each $j=0,\dots, n$, we have
$$e_j(\cLbar_1):=\sup_{\substack{Y\subset H\\ {\rm codim}_H(Y)=j}}\inf_{x\in (H\setminus Y)(\Qbar)} \hhat_{\cLbar_1}(x)$$
and so, indeed hypothesis~(2) of Proposition~\ref{iff proposition} yields that $e_j(\cLbar_1)=0$. In particular, $e_{n}(\cLbar_1)=0$ and thus $\cLbar_1^n=0$. The exact same argument applied for each $i_1,i_2=1,\dots, n$ and for each $m_1,m_2\in\N$ to the metrized line bundle $\cLbar_{i_1,i_2,m_1,m_2}:=\cLbar_1+m_1\cLbar_{\hat F_{i_1}}+m_2\cLbar_{\hat F_{i_2}}$ yields again 
\begin{equation}
\label{cL m i}
\left(\cLbar_1+m_1\cLbar_{\hat F_{i_1}}+m_2\cLbar_{\hat F_{i_2}}\right)^n=0.
\end{equation}
Keeping $i_1$ and $i_2$ fixed and letting $m_1$ and $m_2$ vary in $\N$, we see that equation \eqref{cL m i} yields that $\cLbar_1^{j_0}\cdot \cLbar_{\hat F_{i_1}}^{j_1}\cdot \cLbar_{\hat F_{i_2}}^{j_2}=0$ for each non-negative integers $j_0,j_1,j_2$ such that $j_0+j_1+j_2=n$. Hence   
\begin{equation}
\label{arithmetic Hodge 1}
\cLbar_0^2\cdot \cLbar_1^{n-2}=0;
\end{equation}
moreover, because the numbers $c_i$ satisfy the construction from Claim~\ref{claim c i} (see condition~(A) in the proof of the aforementioned Claim), we also have that
\begin{equation}
\label{Hodge 1}
\cL_0\cdot \cL_1^{n-2}=0. 
\end{equation}
Furthermore, since each $\cLbar_{\hat F_i}$ is semipositive, we obtain that (with the terminology from \cite{Yuan-Zhang-published}) $\cLbar_0$ is $\cLbar_1$-bounded, i.e., there exists $m\in\N$ (any integer larger than $\max_i |c_i|$ would work) such that both $m\cdot \cLbar_1 -\cLbar_0$ and $m\cdot \cLbar_1 + \cLbar_0$ are semipositive.

Since $\cLbar_1$ may not necessarily be arithmetically positive, we alter $\cLbar_1$ by adding to it an arbitrarily positive metrized line bundle $\iota^*(\cC)$ where $\cC$ is a positive metrized line bundle on ${\rm Spec}(\Qbar)$ and $\iota:H\lra {\rm Spec}(\Qbar)$ is the structure morphism (for a similar application, see the proof of \cite[Theorem~4.13]{Yuan-Zhang-published}). Then $\cLbar_0$ would still be $\cLbar_1'$-bounded with respect to this new metrized line bundle $\cLbar_1':=\cLbar_1+\iota^*(\cC)$. Because the generic fiber of $\cLbar_0$ is numerically trivial (according to our choice of the numbers $c_i$ satisfying the conclusion of Claim~\ref{claim c i}), then \eqref{arithmetic Hodge 1} and \eqref{Hodge 1} yield
\begin{equation}
\label{arithmetic Hodge 2}
\cL_0\cdot \left(\cL_1'\right)^{n-1}=0\text{ and }\cLbar_0^2\cdot \left(\cLbar_1'\right)^{n-2}=0.
\end{equation}
Thus the hypotheses of \cite[Theorem~3.2]{Yuan-Zhang-published} are verified and so, we obtain that the metrized line bundle $\cLbar_0$ is itself numerically trivial, i.e., $\hhat_{\cLbar_0}(x)=0$ for each $x\in H(\Qbar)$. This concludes the proof of Lemma~\ref{lemma c i}.  
\end{proof}

So, by Lemma~\ref{lemma c i}, there exist suitable constants $c_i\in\mathbb{R}$ (for $i=1,\dots, n$), not all equal to $0$ such that the metrized line bundle 
$\cLbar_0:=c_1\cdot \cLbar_{\hat F_1} + \cdots + c_n \cLbar_{\hat F_n}\in \overline{{\rm Pic}}(H)\otimes \mathbb{R}$ 
is numerically trivial on $H$ and therefore, for each $\alpha\in H(\Qbar)$, we have that $\cLbar_0\cdot \alpha =0$, i.e.,
\begin{equation}
\label{heights 0}
\sum_{i=1}^n c_i \cdot \hhat_{\cLbar_{\hat F_i}}(\alpha) = 0.
\end{equation}
Since not all $c_i$ are equal to $0$, then there exists some $i_0\in \{1,\dots, n\}$ with the property that 
\begin{equation}
\label{sum c i nonzero}
c_1+ \cdots + c_{i_0-1} + c_{i_0+1} + \cdots + c_n \ne 0.
\end{equation}
Now, for any $\alpha:=(a_1,\dots, a_n)\in H(\Qbar)$ and for any $i=1,\dots, n$, we have that 
\begin{equation}
\label{height cL alpha}
\hhat_{\cLbar_{\hat F_i}}(\alpha) = \hhat_{f_1}(a_1) + \cdots + \hhat_{f_{i-1}}(a_{i-1}) + \hhat_{f_{i+1}}(a_{i+1}) + \cdots + \hhat_{f_n}(a_n),
\end{equation}
as shown in \eqref{equation height cL i}. Now, if   
$$
\hhat_{f_1}(a_1)=\cdots =\hhat_{f_{i_0-1}}\left(a_{i_0-1}\right) = \hhat_{f_{i_0+1}}\left(a_{i_0+1}\right) = \cdots = \hhat_{f_n}(a_n)=0,
$$
then \eqref{heights 0}, \eqref{sum c i nonzero} and \eqref{height cL alpha}  yield that also $\hhat_{f_{i_0}}\left(a_{i_0}\right)=0$, as claimed in the conclusion of Proposition~\ref{iff proposition}. This concludes our proof of Proposition~\ref{iff proposition} assuming each rational function $f_i$ along with the hypersurface $H$ are defined over $\Qbar$.
\end{proof}

\begin{proof}[Proof of Proposition~\ref{iff proposition} assuming hypothesis~(1) holds.]
We let $K\subset \C$ be a finitely generated extension of $\Qbar$ such that each $f_i\in K(x)$ and also $H$ is defined over $K$. We argue by induction on $r:={\rm trdeg}_{\Qbar}K$; the case $r=0$ is already proved using Proposition~\ref{iff proposition} with hypothesis~(2). Hence, we assume the conclusion of Proposition~\ref{iff proposition} holds whenever $r<s$ (for some $s\in\N$) and we prove that it also holds when $r=s$. We know there exists an infinite sequence $S$ of points $\alpha_j\in H(\C)$ such that $\alpha_j$ has its $i$-th coordinate preperiodic under the action of $f_i$ (for each $i=1,\dots, n$). Also, we let 
$$d:=\deg(f_1)=\deg(f_2)=\cdots = \deg(f_n).$$ 
Then we let $K_0$ be a subfield $\Qbar\subset K_0\subset K$ such that ${\rm trdeg}_{K_0}K=1$ and we let $\cC$ be a curve defined over $K_0$ whose function field is $K$ (at the expense of replacing both $K_0$ and $K$ by finite extensions, we may assume $\cC$ is a projective, smooth, geometrically irreducible curve). We fix some algebraic closures $\overline{K_0}\subset \Kbar$ of our fields. 

There exists a Zariski dense, open subset $C\subset \cC$ such that we may view each $f_i$ as a base change of an endomorphism $f_{i,C}$  of $\bP^1_C$; similarly, $H$ is the base change of a hypersurface $H_C\subset (\bP^1_C)^n$, while $S$ is the base change of a subset $S_C\subset H_C$. For each geometric point $t\in C\left(\overline{K_0}\right)$, the objects $H_C$, $f_{i,C}$ and $S_C$ have reductions $H_t$, $f_{i,t}$ and respectively $S_t$, such that $S_t\subset H_C$ consists of points with their $i$-th  coordinate preperiodic under the action of $f_{i,C}$, for each $i=1,\dots, n$. 

\begin{claim}
\label{specialization claim}
There exists a Zariski dense, open subset $C_0\subset C\subset \cC$ such that for each $t\in C_0\left(\overline{K_0}\right)$, the set $S_t$ is Zariski dense in $H_t$. 
\end{claim}

\begin{proof}[Proof of Claim~\ref{specialization claim}.]
We let $\Psi:=(f_1,\dots, f_{n-1})$ be the coordinatewise action of these rational functions on the first $n-1$ coordinates of $(\bP^1)^n$; since $d_1=\cdots = d_{n-1}=d>1$, we know that $\Psi$ is a polarizable endomorphism of $(\bP^1)^{n-1}$. We let $\tilde{S}$ be the projection of the set $S$ on the first $n-1$ coordinate axes of $(\bP^1)^n$; because $S\subset H$ is dense and $H$ projects dominantly onto the first $n-1$ coordinate axes, we conclude that $\tilde{S}\subset (\bP^1)^{n-1}$ is also dense. Note that each point of $\tilde{S}$ is a preperiodic point for $\Psi$. As before, we let $\tilde{S}_t$ be the specialization of the set $\tilde{S}_C$ at some point $t\in C_0\left(\overline{K_0}\right)$.

As proven in   \cite[Theorem~4.7]{Yuan-Zhang-2} (see also \cite[Lemma~3.2.3]{Yuan-Zhang-1}), the set $\tilde{S}_t\subset (\bP^1)^{n-1}$ is still Zariski dense for all the $\overline{K_0}$-points $t$ of a dense open subset $C_0\subseteq C$. Here it is the only point in our argument where we use that $d_1=\cdots = d_n$ because Yuan-Zhang \cite{Yuan-Zhang-2, Yuan-Zhang-1} show that specializing a Zariski dense set of preperiodic points for a \emph{polarizable} endomorphism yields also a Zariski dense set of preperiodic points for all specializations in a dense, open subset of the base; in their proof, they employ a result of Faber \cite{Faber} and of Gubler \cite{Gubler} regarding the equidistribution of subvarieties of a given polarizable dynamical system $(X,\Phi)$ with respect to the invariant measure of $\Phi$. (As an aside, we note that the results of \cite{Yuan-Zhang-1} were recently published in \cite{Yuan-Zhang-published}, while \cite{Yuan-Zhang-2} has been  
updated to \cite{Yuan-Zhang-arxiv} using slightly different arguments.) Finally, since $\tilde{S}_t\subset (\bP^1)^{n-1}$ is Zariski dense, then the Zariski closure of $S_t$ must have dimension $n-1$ because $S_t$ projects to $\tilde{S}_t$ on the first $n-1$ coordinate axes of $(\bP^1)^n$. Hence $S_t\subset H_t$ is Zariski dense, which concludes the proof of Claim~\ref{specialization claim}. 
\end{proof}

Let $C_0$ be the Zariski dense, open subset of $C$ satisfying the conclusion of Claim~\ref{specialization claim}. At the expense of perhaps shrinking $C_0$ to a smaller, dense, open subset, we may assume that 
\begin{equation}
\label{same degree}
\deg\left(f_{i,t}\right)=d>1\text{ for all $i=1,\dots, n$ and each $t\in C_0\left(\overline{K_0}\right)$.}
\end{equation}
For each $t\in C_0(\Qbar)$, our inductive hypothesis  (which can be applied since each $f_i$ and also $H$ are defined over $\overline{K_0}$ and ${\rm trdeg}_{\Qbar}\overline{K_0}<s$) yields the existence of some index $i_t\in\{1,\dots, n\}$ which has the property that for each $\alpha\in H_t(\Qbar)$, if we know that the $j$-th coordinate of $\alpha$ is preperiodic under the action of $f_{j,t}$ for each $j\in\{1,\dots, n\}\setminus\{i_t\}$, then also the $i_t$-th coordinate of $\alpha$ is preperiodic under the action of $f_{i_t,t}$. 

Let $h_\cC(\cdot )$ be a height function for the points on $\cC\left(\overline{K_0}\right)$ corresponding to a divisor of degree $1$ on $\cC$, constructed with respect to the Weil height on $\overline{K_0}$. Note that if ${\rm trdeg}_{\Q}K_0\ge 1$, then we construct the Weil height on the function field $K_0/\Qbar$ as in \cite{Bombieri}. At the expense of replacing $C_0$ by an infinite  subset $U_0$ for which  
\begin{equation}
\label{height C t}
\sup_{t\in U_0} h_\cC(t)=+\infty, 
\end{equation}
we may even assume that for each $t\in U_0$, there is the same index $i_0:=i_t\in\{1,\dots, n\}$ satisfying the above  property. We show next that this index $i_0$ would satisfy the conclusion of Proposition~\ref{iff proposition} for $H$.

Indeed, let $\alpha=(a_1,\dots, a_n)\in H(\Kbar)$ with the property that for each $j\in\{1,\dots, n\}\setminus \{i_0\}$, we have that $a_j$ is preperiodic under the action of $f_j$. Then for each $t\in U_0$ we have that each $a_{j,t}$ (for $j\in\{1,\dots, n\}\setminus\{i_0\}$) is preperiodic for $f_{j,t}$ and so, also  $a_{i_0,t}$ is preperiodic under the action of $f_{i_0,t}$. Therefore, the canonical height 
\begin{equation}
\label{canonical height 111}
\hhat_{f_{i_0,t}}\left(a_{i_0,t}\right)=0, 
\end{equation}
where $\hhat_{f_{i_0,t}}(\cdot )$ is the canonical height corresponding to the rational function $f_{i_0,t}$ (which has degree larger than $1$; see \eqref{same degree}), constructed using the Weil height on $\overline{K_0}$. Using \cite[Theorem~4.1]{C-S}, we have that 
\begin{equation}
\label{variation of canonical height}
\lim_{h_\cC(t)\to\infty} \frac{\hhat_{f_{i_0,t}}\left(a_{i_0,t}\right)}{h_\cC(t)} = \hhat_{f_{i_0}}\left(a_{i_0}\right), 
\end{equation}
where $\hhat_{f_{i_0}}(\cdot )$ is the canonical height of $f_{i_0}$ constructed with respect to the function field $K/K_0$. Equations \eqref{height C t},  \eqref{canonical height 111} and \eqref{variation of canonical height} yield that 
\begin{equation}
\label{height is 0}
\hhat_{f_{i_0}}\left(a_{i_0}\right) = 0.
\end{equation}
If $f_{i_0}\in K(x)$ is not isotrivial over $K_0$, then \cite{Baker} (see also \cite{Rob} for the case of polynomials) yields that \eqref{height is 0} is equivalent with saying that $a_{i_0}$ is preperiodic under the action of $f_{i_0}$, as desired. Now, if $f_{i_0}$ is isotrivial over $K_0$, then there exists a linear transformation 
$$\nu:\bP^1\lra \bP^1\text{ (defined over $\Kbar$)}$$ 
such that $\nu^{-1}\circ f_{i_0}\circ \nu\in \overline{K_0}(x)$. If $\nu^{-1}(a_{i_0})\in \overline{K_0}$, then since we know there exists even a single specialization $t$ such that $a_{i_0,t}$ is preperiodic for $f_t$, we get that also $a_{i_0}$ is preperiodic for $f_{i_0}$. On the other hand, if $\nu^{-1}(a_{i_0})\notin \overline{K_0}$, then $\nu^{-1}(a_{i_0})$ cannot be preperiodic for $\nu^{-1}\circ f_{i_0}\circ \nu\in \overline{K_0}(x)$ and so, $a_{i_0}$  is not preperiodic for $f_{i_0}$, contradiction. This concludes the proof of Proposition~\ref{iff proposition} under hypothesis~(1). 
\end{proof}

  
  \section{Hypersurfaces having a Zariski dense set of preperiodic points}
\label{section from identical to preperiodic}


In this section, we prove Theorem \ref{measure to periodicity}, which (essentially) says  that there is no hypersurface $H$ containing a Zariski dense set of preperiodic points under the coordinatewise action of some rational functions $f_i$, along with some additional technical conditions.  To make things simple, we work on a hypersurface $H\subset (\P^1)^{n+1}$ of dimension $n$ and use the following notation 
  $$\tilde x=(x_1, \cdots, x_n), ~ \underline x=(x_1, \cdots, x_{n-1})$$
and hence $\tilde a=(a_1, \cdots, a_n), ~ \underline a=(a_1, \cdots, a_{n-1})$,  etc. We denote by $D(a,r)\subset \C$ the usual disk of radius $r$ centered at $a$; also, we use the following notation for polydiscs:
 $$D_{n-1}(\underline a, r)=D(a_1, r)\times \cdots D(a_{n-1}, r)\text{ and } D_{n}(\tilde a, r)=D_{n-1}(\underline a, r)\times D(a_n, r).$$

For the benefit of our readers, we split our proof of Theorem~\ref{measure to periodicity} in several subsections, each one presenting a different step in our argument. 


\subsection{Statement of our theorem}


\begin{theorem}\label{measure to periodicity}
Let $n\ge 2$, let  $f_i$ be rational functions defined over $\C$ of degrees $d_i>1$ (for $1\leq i\leq n+1$), and let $H\subset (\mathbb{P}^1)^{n+1}$ be an irreducible hypersurface defined over $\C$ which projects dominantly onto each subset of $n$ coordinate axes. For each $i=1,\dots, n+1$, let $\tilde f_i$ be the coordinatewise action on $(\bP^1)^n$ given by  
$$(x_1,\dots, x_{i-1},x_{i+1},\cdots ,x_{n+1})\mapsto \left(f_1(x_1),\cdots , f_{i-1}(x_{i-1}),f_{i+1}(x_{i+1}),\cdots , f_{n+1}(x_{n+1})\right).$$
Let $\hat{\mu}_i$ be the measures on $H$ induced from the dynamical systems $\left((\bP^1)^n, \tilde f_i\right)$ and assume that $\hat{\mu}_i=\hat{\mu}_{n+1}$ for $1\leq i\leq n$. Also assume that there is a point $(\tilde a, b_0)\in H\cap \C^{n+1}$ with $\tilde a=(a_1, \cdots, a_n)$, such that: 
\begin{itemize}
\item $a_i$ is a repelling fixed  point of $f_i$ for $1\leq i\leq n$; and 
\item $b_1:=f_{n+1}(b_0)$ is a fixed point of $f_{n+1}$; and
\item there is a holomorphic germ $h(\tilde x)$ at $\tilde a$ with $h(\tilde a)=b_0$,  and $(\tilde x, h(\tilde x))\in H(\C)$ for all $\tilde x\in \C^n$ in a small (complex analytic) neighbourhood of $\tilde a$. Moreover, for each $i=1, \cdots, n$ we have that 
   $$\beta_i:=\frac{\partial h}{\partial x_i}(\tilde a)\neq 0.$$
\end{itemize}
Then the $f_i$'s must be exceptional, and moreover, they are 
\begin{itemize}
\item either all of them conjugate to monomials and $\pm$Chebyshev polynomials,
\item or all of them Latt\'es maps.
\end{itemize}
\end{theorem}

\begin{proof}
As we previously stated, we will prove Theorem~\ref{measure to periodicity} over the next several subsections of Section~\ref{section from identical to preperiodic}.


\subsection{Julia sets and invariant measures}


From the assumptions of Theorem \ref{measure to periodicity}, the multiplier 
   $$\lambda_i:=f_i'(a_i)$$
has absolute value $|\lambda_i|>1$ for $1\leq i\leq n$.  So, each $a_i$ is in the support of the Julia set $J_{f_i}$ of $f_i$, for $i=1,\dots, n$. Thus $(\tilde a,b_0)$ is in the support of $\hat \mu_{n+1}$ and because $\hat \mu_n=\hat \mu_{n+1}$, we get that $(\tilde a, b_0)$ must be in the support of $\hat \mu_n$. Therefore,   $b_0$ must be in the support $J_{f_{n+1}}$ of $\mu_{f_{n+1}}$. Hence $b_1=f_{n+1}(b_0)\in J_{f_{n+1}}$ and so, it has multiplier 
  $$\rho:=f'_{n+1}(b_1)$$
of absolute value $|\rho|\geq 1$. Let $j_0$ be the local degree of the map $f_{n+1}(x)$ at $x=b_0$, and let $g(x)$ be a holomorphic germ  on $\bP^1$ at $b_0$ which is one of the following branches
\begin{equation}\label{def of g}
g(x):=f_{n+1}^{-1}\circ f_{n+1}\circ f_{n+1}(x)
\end{equation}
satisfying $g(b_0)=b_0$. 
Although there are $j_0$ different choices for $g(x)$, in the rest of this section we fix our choice $g(x)$ for such a branch. An easy computation shows that 
\begin{equation}\label{lambda}
\lambda:=g'(b_0)=\sqrt[j_0]{\rho}
\end{equation}
is a $j_0$-th root of the multiplier $\rho$ of $f_{n+1}$ at $b_1$. Since $\mu_{f_{n+1}}$ admits no atoms on $\bP^1$ and $\mu_{f_{n+1}}(f_{n+1}(A))=d_{n+1}\cdot\mu_{f_{n+1}}(A)$ for any Borel set $A$ with $f_{n+1}$ being injective on $A$, the definition of $g(x)$ yields that  
   $$\mu_{f_{n+1}}(g(A))=d_{n+1}\cdot \mu_{f_{n+1}}(A)$$
for any Borel set $A$ in a small neighborhood of $b_0$. 

\begin{lemma}\label{greater than one}
The multiplier $\lambda$ of $g(x)$ at $b_0$ has absolute value $|\lambda|>1$. 
\end{lemma}

\begin{proof}[Proof of Lemma \ref{greater than one}]We first assume that $|\lambda|\leq 1$ and then prove the lemma by deriving a contradiction. Using \eqref{lambda} and the fact that $|\rho|\geq 1$, we get that $|\lambda|=1$. 

Pick a positive integer $m$ with $d_n<d^m_{n+1}$. Let 
\begin{equation}\label{phi00}
\Phi_{00}(\tilde x):=(\underline x, h(\underline x, f_n(x_n)))\textup{ and } \Phi_{11}(\tilde x):=(\underline x, g^m\circ h(\tilde x))
\end{equation}
be functions locally defined in a neighborhood of $\tilde a\in \C^n$, mapping that small neighborhood of $\tilde a$ into a neighborhood of $(\underline a, b_0)\in \C^n$. Since $\hat \mu_n=\hat\mu_{n+1}$, there exists some $c>0$ with
\begin{equation}\label{lambda contradiction}\Phi_{00}^*(\tilde \mu_{n})=c\cdot d_{n}\cdot \tilde \mu_{n+1}\textup{ and } \Phi_{11}^*(\tilde \mu_{n})=c\cdot d^m_{n+1}\cdot \tilde \mu_{n+1}.
\end{equation}
The measures $\tilde \mu_{n}$ and $\tilde \mu_{n+1}$  (defined in \eqref{measure tilde i}) appearing in \eqref{lambda contradiction} are restricted on some small  neighborhood of $\tilde a$ (respectively of $(\underline a, b_0)$). Let $A$ be the polydisc given by $A:=D_{n-1}(\underline a, r_1)\times D(a_n, r_2)$ for very small $r_2$ and much smaller $r_1$.  We claim that $\Phi_{11}(A)\subset \Phi_{00}(A)$. To see this, let $r_2$ be very small and we see that $f_n(D(a_n, r_2))\sim D(a_n,|\lambda_n|r_2)$. As $|\lambda|=|g'(b_0)|=1<|\lambda_n|$ and $\beta_n=\frac{\partial h}{\partial x_n}(\tilde a)\neq 0$, using \eqref{phi00} we can pick some very small $r_2$ and a much smaller $r_1$ such that 
\begin{equation}
\label{equation A}
\Phi_{11}(A)\subset D_{n-1}(\underline a, r_1)\times D(b_0, r_2\cdot |\beta_n|\cdot|\lambda_n|^{1/2})\subset \Phi_{00}(A).
\end{equation}
However, combining \eqref{lambda contradiction} with $d_n<d^m_{n+1}$ gives
  $$\tilde \mu_n(\Phi_{11}(A))>\tilde\mu_n( \Phi_{00}(A)),$$
which is a contradiction. This concludes the proof of Lemma~\ref{greater than one}.
\end{proof}


\subsection{A special sequence of tuples of positive integers}


Now since $|\lambda|>1$ and $|\lambda_i|>1$ for $1\leq i\leq n$, we can pick a sequence of tuples  of positive integers $(j_\ell, j_{1,\ell}, \cdots, j_{n,\ell})$ such that $j_\ell\to \infty$ as $\ell\to \infty$ and moreover,  
\begin{equation}\label{lambdas}
\lim_{\ell\to \infty}\inf\frac{|\lambda_1^{j_{1,\ell}}|}{|\lambda^{j_\ell}|}, \cdots,  \lim_{\ell\to \infty}\inf\frac{|\lambda_{n-1}^{j_{1,\ell}}|}{|\lambda^{j_\ell}|}\geq \lim_{\ell\to \infty}\frac{\lambda_n^{j_{n,\ell}}}{\lambda^{j_\ell}}=1.
\end{equation}

It will be useful later in our argument (see Lemma~\ref{first invariant}) that our sequence of tuples $(j_\ell, j_{1,\ell}, \cdots, j_{n,\ell})$ satisfies the following  arithmetic property in addition to \eqref{lambdas}. We want that for every $N\in\N$, there exist $\ell_2>\ell_1>N$ such that 
\begin{equation}
\label{arithmetic condition}
j_{\ell_2}=j_{\ell_1}\text{ and }j_{i,\ell_2}=j_{i,\ell_1}\text{ for $2\le i\le n$, while }j_{1,\ell_2}=j_{1,\ell_1}+1.
\end{equation}
In order to achieve \eqref{arithmetic condition}, we may replace the original  sequence of tuples $\left\{(j_\ell,j_{1,\ell},\dots, j_{n,\ell})\right\}_{\ell=1}^\infty$ 
by the larger sequence $\left\{(j'_\ell,j'_{1,\ell},\dots, j'_{n,\ell})\right\}_{\ell=1}^\infty$ 
for which
$$j'_{2\ell-1}=j'_{2\ell}=j_{\ell}\text{ and }j'_{i,2\ell-1}=j'_{i,2\ell}=j_{i,\ell}$$
for $i=2,\dots, n$, while
$$j'_{1,2\ell-1}=j_{1,\ell}\text{ and }j'_{1,2\ell}=j_{1,\ell}+1$$
and still the new sequence $\left\{(j'_\ell,j'_{1,\ell},\dots, j'_{n,\ell})\right\}_{\ell=1}^\infty$ satisfies \eqref{lambdas} and the fact that $j_{\ell}\to\infty$ as $\ell\to\infty$. For the sake of simplifying our notation, we will denote our new sequence of tuples also as $\left\{(j_\ell, j_{1,\ell},\dots, j_{n,\ell})\right\}_{\ell=1}^\infty$, but we note that this sequence of tuples satisfies \eqref{arithmetic condition}. 


\subsection{Local symmetries for the Julia sets}


From \cite{Milnor:book}, we know we can conjugate $f_i$ (for $1\leq i\leq n$) and $g$ to linear maps in small neighborhoods of the repelling fixed points.   More precisely, there exist holomorphic germs $\phi_i$ at $x=0$ satisfying 
$$\phi_i(0)=a_i\text{, }\phi_{n+1}(0)=b_0\text{, }\phi_i'(0)=\phi_{n+1}'(0)=1\text{ for $1\le i\le n$ and}$$  
$$\phi_i^{-1}\circ f_i\circ \phi_i(x)=\lambda_i \cdot x\text{ for $1\le i\le n$,}$$
$$\text{while }\phi_{n+1}^{-1}\circ g\circ \phi_{n+1}(x)=\lambda\cdot  x.$$
We notice that for $(x_1, \cdots, x_n)$ in a neighborhood of $\tilde a\in \C^n$, we have an equality of germs:  
$$
g^{j_\ell}\circ h\circ \left(f_1^{-j_{1,\ell}}(x_1), \cdots, f_n^{-j_{n,\ell}}(x_n)\right)$$
$$ = \phi_{n+1}\circ \left(\lambda^{j_\ell}\cdot h_{\phi}\left(\frac{\phi_1^{-1}(x_1)}{\lambda_1^{j_{1,\ell}}}, \cdots, \frac{\phi_n^{-1}(x_n)}{\lambda_n^{j_{n,\ell}}}\right)\right),$$
where $h_\phi:=\phi_{n+1}^{-1}\circ h\circ (\phi_1, \cdots, \phi_n)$ and $f_i^{-1}$ is the germ of a branch of the inverse of $f_i$ at $x_i=a_i$ with $f_i^{-1}(a_i)=a_i$. So, using also \eqref{lambdas}, then for very small $r_0>0$ and all $\tilde x$ in the ball $B(\tilde a, r_0)\subset \C^n$ of radius $r_0$, the map 
$$\tilde x\mapsto g^{j_\ell}\circ h\circ (f_1^{-j_{1,\ell}}, \cdots, f_n^{-j_{n,\ell}})(\tilde x)$$ 
is well defined and uniformly bounded on $B(\tilde a, r_0)$ for all $\ell\geq 1$. Next, we construct the function  
  $$\Psi(\tilde x):=(x_1, \cdots, x_{n-1}, h(\tilde x))$$
for $\tilde x=(x_1, \cdots, x_n)$, which is locally one-to-one at $\tilde x=\tilde a$ since $\beta_n=\frac{\partial h}{\partial x_n}(\tilde a)\neq 0$. Shrinking $r_0$ if necessary, we let
\begin{equation}
\begin{split}
\label{split equation}
\Psi_\ell(\tilde x)&:=\Psi^{-1}\circ\left(x_1, \cdots, x_{n-1}, g^{j_\ell}\circ h\circ (f_1^{-j_{1,\ell}}, \cdots, f_n^{-j_{n,\ell}})(\tilde x)\right)\\
&:=(x_1, \cdots, x_{n-1}, h_\ell(\tilde x))
\end{split}
\end{equation}
for all $\tilde x\in B(\tilde a, r_0)$ and all $\ell\geq 1$, where $h_\ell$ is some local analytic function on $B(\tilde a, r_0)$ satisfying \eqref{split equation}.
 
\begin{lemma}\label{normal family}
The family of functions $\{h_\ell(\tilde x)\}_{\ell\geq 1}$ restricted on $B(\tilde a, r_0)$ is a normal family. 
\end{lemma}

\begin{proof}[Proof of Lemma \ref{normal family}] 
Since $\tilde x\mapsto g^{j_\ell}\circ h\circ (f_1^{-j_{1,\ell}}, \cdots, f_n^{-j_{n,\ell}})(\tilde x)$ is uniformly bounded on $B(\tilde a, r_0)$ for all $\ell\geq 1$, then 
 that $h_\ell$ (defined as in \eqref{split equation}) is uniformly bounded on $B(\tilde a, r_0)$, i.e., there exist $R>0$ such that 
  $$h_\ell(B(\tilde a, r_0))\subset B(b_0, R)\subset \C$$
for all $\ell\geq 1$. Hence $h_\ell$ is a distance non-increasing map from $B(\tilde a, r)$ (with respect to the Bergman metric) to $B(b_0, R)$ (with respect to the hyperbolic metric). Thus $\{h_\ell(\tilde x)\}_{\ell\geq 1}$ is equicontinuous on $B(\tilde a, r_0)$, or equivalently, $\{h_\ell(\tilde x)\}_{\ell\geq 1}$ is a normal family. 
\end{proof}

From Lemma \ref{normal family}, we can pick a subsequence of $\{\Psi_\ell\}_{\ell\geq 1}$ which converges uniformly on $B(\tilde a, r_0)$. By passing to a subsequence, without loss of generality, we can assume that the sequence $\{\Psi_\ell\}_{\ell\geq 1}$ itself converges uniformly to 
  $$\Psi_0(\tilde x):=\lim_{\ell\to \infty} \Psi_\ell(\tilde x)$$
and satisfies \eqref{arithmetic condition} with $\Psi_0(\tilde a)=\tilde a$ and $\Psi_0(\tilde x)=:(x_1, \cdots, x_{n-1}, h_0(\tilde x))$. Since
  $$\frac{\partial h_0}{\partial x_n}(\tilde a)=\lim_{\ell\to \infty} \frac{\partial h_\ell}{\partial x_n}(\tilde a)=\lim_{\ell\to \infty} \frac{\lambda^{j_\ell}}{\lambda_n^{j_{n,\ell}}}=1\neq 0,$$
the map $\Psi_0$ is locally one-to-one at $\tilde x=\tilde a$. Shrinking $r_0$ if necessary, we can further assume that the sequence of maps 
  $$\Psi_0^{-1}\circ \Psi_\ell(\tilde x)=:(x_1, \cdots, x_{n-1}, \hbar_\ell(\tilde x))=:\Phi_\ell(\tilde x)$$
converges uniformly to the identity map on $B(\tilde a, r_0)$ as $\ell\to \infty$. The next goal is to show that $\Phi_\ell$ is the identity map for all large $\ell$; see Lemma \ref{identical maps}.


\subsection{Equal currents}

  
\begin{proposition}\label{equal currents} Let $r_1$ and $r_2$ be positive real numbers and let $u_1, \cdots, u_n$ and $u$ be continuous subharmonic functions on $D(0, r_1)$, respectively on $D(0, r_2)$.  Let $\theta$ be a holomorphic map from $D_n(\tilde 0, r_1)$ to $D(0,r_2)$ and moreover, assume the following two  $(n,n)$-currents satisfy the relation: 
  $$dd^cu_1(x_1)\wedge\cdots \wedge dd^c u_n(x_n)=c_0\cdot dd^cu_1(x_1)\wedge\cdots \wedge dd^c u_{n-1}(x_{n-1})\wedge dd^c u\circ \theta(\tilde x)$$
on $D_n(\tilde 0, r_1)$ for some constant $c_0>0$. Then for any given point $\underline \alpha$ in the support of $dd^cu_1(x_1)\wedge\cdots \wedge dd^c u_{n-1}(x_{n-1})$, we have the following equality of $(1,1)$-currents on $D(0, r_1)$:
   $$dd^c u_n(x_n)=c_0\cdot dd^c u\circ \theta(\underline \alpha, x_n).$$ 
\end{proposition}

\begin{proof}[Proof of Proposition \ref{equal currents}] Let $\underline \alpha$ be a point in the support of $dd^cu_1(x_1)\wedge\cdots \wedge dd^c u_{n-1}(x_{n-1})$. 
It suffices to show that for any $C^\infty$ real function $\varphi$ with compact support on $D(0, r_1)$, one has 
    $$\int_{D(0, r_1)} \varphi(x_n) dd^c u_n(x_n)=c_0\int_{D(0, r_1)} \varphi (x_n)dd^c u\circ \theta(\underline \alpha, x_n).$$
 To see this, we let $\underline \mu$ be the measure on $D_{n-1}(\underline 0, r_1)$ with 
   $$d \underline \mu(\underline x):=c_0\cdot dd^cu_1(x_1)\wedge\cdots \wedge dd^c u_{n-1}(x_{n-1})$$
and let $\tilde \mu$ be the measure on $D_n(\tilde 0, r_1)$ with 
  $$d\tilde \mu(\tilde x):=d \underline \mu(\underline x) \wedge \frac{dx_n\wedge d\bar x_n}{-4\pi i}$$

For each small positive real number $r$, we let  $\eta_{r}(\underline x)$ be a $C^\infty$-function on $D_{n-1}(\underline 0, r_1)$ satisfying the properties:
\begin{itemize}
\item  $0\leq \eta_{r}\leq 1$;
\item $\eta_r$ is supported on $D_{n-1}(\underline \alpha, r)$; and 
\item $\eta_r=1$ on  $D_{n-1}(\underline \alpha, r/2)$. 
\end{itemize}
From the proportionality assumption of the two $(n,n)$-currents, we get
\begin{equation}
\begin{split}
\frac{1}{c_0}\left(\int \eta_{r} d\underline\mu\right) \int\varphi dd^c u_n&=\frac{1}{c_0}\int \eta_{r}(\underline x)\varphi(x_n)d\underline \mu (\underline x) \wedge dd^c u_n(x_n)\\
&=\int \eta_{r}(\underline x)\varphi(x_n)d\underline \mu(\underline x) \wedge dd^c u\circ \theta(\tilde x)\\
&=\int u\circ \theta(\tilde x)d\underline \mu \wedge dd^c (\eta_{r}\varphi)\\
&=\int \eta_{r}(\underline x) u\circ \theta(\tilde x) \triangle \varphi(x_n) d\tilde \mu(\tilde x)
\end{split}
\end{equation}
where $\triangle$ is the Laplacian and the right hand side is integrated over the domain $D_{n}(\tilde 0, r_1)$. Similarly we derive that 
\begin{equation*}
\left(\int \eta_{r} d\underline\mu\right) \int \varphi dd^c u\circ \theta(\underline \alpha, x_n)=\int \eta_{r}(\underline x) u\circ \theta(\underline \alpha, x_n) \triangle \varphi(x_n) d\tilde \mu(\tilde x).
\end{equation*}
Now let 
   $$\Theta_{r}(\tilde x):= \eta_{r}(\underline x)\cdot \left( u\circ \theta(\underline \alpha, x_n) - u\circ \theta(\tilde x)\right)\cdot  \triangle \varphi(x_n)$$
which is supported on $D_{n-1}(\underline \alpha, r)\times D(0, r_1)$. Hence as $u\circ \theta$ is continuous and $\varphi$ has compact support on $D(0, r_1)$, there exist constants $\epsilon_r\to 0$ as $r\to 0$ such that for any $\tilde x\in D_n(\tilde 0,r_1)$, we have 
   $$|\Theta_{r}(\tilde x)|\leq \eta_{r}(\underline x)\cdot \epsilon_r.$$
Consequently 
   $$\left|\frac{1}{c_0}\int\varphi dd^c u_n- \int \varphi dd^c u\circ \theta(\underline \alpha, x_n)\right|\leq \frac{ \int_{D_n(\tilde 0, r_1)} \eta_{r}(\underline x)\cdot \epsilon_r~ d\tilde \mu(\tilde x)}{\int_{D_{n-1}(\underline 0, r_1)}\eta_{r}(\underline x) d\underline\mu(\underline x)}= \epsilon_r\cdot c_1$$
with $c_1=\int_{D(0, r_1)} 1\cdot \frac{dx_n\wedge d\bar x_n}{-4\pi i}$. Now letting $r\to 0$, the conclusion in Proposition~\ref{equal currents} follows. 
\end{proof}


\subsection{The rational functions must be exceptional}


The next result yields half of the conclusion in Theorem~\ref{measure to periodicity} by showing that if $f_{n+1}$ is an exceptional rational function, then each $f_i$ is exceptional, and moreover, each $f_i$ is either Latt\'es or not, depending on whether $f_{n+1}$ is a Latt\'es map, or not.

\begin{corollary}\label{non-exceptional}
The following statements hold:
\begin{itemize}
\item if $f_{n+1}$ is conjugate to a monomial or a $\pm$Chebyshev polynomial, then each $f_i$ (for $i=1,\dots, n$) is conjugate to a monomial or a  $\pm$Chebyshev polynomial.  
\item if $f_{n+1}$ is a Latt\'es map, then each $f_i$ is a Latt\'es map.
\end{itemize}
\end{corollary}

\begin{proof}[Proof of Corollary \ref{non-exceptional}] 
So, we assume that $f_{n+1}$ is exceptional. 
Without loss of generality, we show that $f_n$ is exceptional as well and moreover, it is Latt\'es if and only if $f_{n+1}$ is a Latt\'es map. Since $f_i$ (and $f_{n+1}$) has continuous potential near $a_i$ (respectively near $b_0$) and moreover, $a_i\in J_{f_i}$ which is the support of $\mu_{f_i}$, then Proposition~\ref{equal currents} along with the hypotheses of Theorem~\ref{measure to periodicity} yield that the map $h(\underline a, \cdot)$ which sends a neighborhood of $a_n\in J_{f_n}$ to a neighborhood of $b_0\in J_{f_{n+1}}$ preserves the measures up to a scaling, i.e., for some $c>0$
 \begin{equation}\label{local no exception}
 h^*(\underline a, \cdot)\mu_{f_{n+1}}=c\cdot \mu_{f_{n}}.
\end{equation}
In \cite[Theorem~1]{Levin}, it was shown that there exists an infinite nontrivial family of symmetries on $J_f$ if and only if $f$ is post-critically finite with parabolic orbifold; hence \eqref{local no exception} (see also Subsection~\ref{sym}) yields that $f_n$ must be exceptional. 

By a theorem of  Zdunik \cite{Zdu},  a rational function $f$  is Latt\`es if and only if $J_f$ is $\P^1$ and $\mu_f$ is absolutely continuous with respect to Lebesgue measure on $\P^1$; therefore, \eqref{local no exception} yields that  $f_n$ is Latt\'es if $f_{n+1}$ is Latt\'es.

Assume that $f_{n+1}$ is conjugate either to a monomial or $\pm$Chebyshev polynomial. Then \eqref{local no exception} yields that $J_{f_n}$ is a one-dimensional topological space of Hausdorff dimension $1$. 
According to Hamilton \cite{Ham}, a Julia set which is a one-dimensional topological manifold must be either a circle, closed line segment (up to an automorphism of $\P^1$) or of Hausdorff dimension greater than one; thus $J_{f_n}$ is itself a circle or a closed line segment (up to an automorphism of $\P^1$). This yields that $f_n$ must be conjugated to a monomial or a $\pm$Chebyshev polynomial, which concludes the proof of Corollary~\ref{non-exceptional}.  
\end{proof}


\subsection{Conclusion of our arguments}


Corollary~\ref{non-exceptional} yields that all we have left to prove in Theorem~\ref{measure to periodicity} is that $f_{n+1}$ must be exceptional. So, from now on, we assume that $f_{n+1}$ is non-exceptional and we will derive a contradiction.

\begin{lemma}\label{identical symmetry}
Let $\mathcal S$ be the family of symmetries of $J_{f_{n}}$ on $D(a_n, r)$ for some $r>0$. Then there exists $\epsilon>0$ such that for any $\zeta\in \mathcal S$ with 
   $$\sup_{x\in B(a_n, r)} |\zeta(x)-x|<\epsilon,$$
we must have $\zeta(x)\equiv x$ for  $x\in D(a_n, r)$. 
\end{lemma}

\begin{proof}[Proof of Lemma \ref{identical symmetry}] Suppose this lemma is not true, then there exists a sequence of integers $\epsilon_\ell>0$ with $\epsilon_\ell\to 0$ as $\ell$ tends to infinity, and a sequence of functions  $\zeta_\ell\in \mathcal S$, which are not the identity map, such that 
\begin{equation}
\label{epsilon l}
\sup_{x\in D(a_n, r)}|\zeta_\ell(x)-x|=\epsilon_\ell.
\end{equation}
Consequently, $\{\zeta_\ell(x)\}_{\ell\geq 1}$ is a normal family with no subsequence having a constant limit (because $\zeta_\ell$ tends to the identity map as $\ell\to \infty$). By Levin's result \cite{Levin},  $\{\zeta_\ell\}_{\ell\geq 1}$ must consist of finitely many elements, which is a contradiction because there are infinitely many distinct real numbers  $\epsilon_\ell$ as in \eqref{epsilon l}. 
\end{proof}

\begin{lemma}\label{identical maps}
There exists  $N\in \N$, such that $\Phi_\ell$ is the identity map on $B(\tilde a, r_0)$ for all $\ell\geq N$. 
\end{lemma}

\begin{proof}[Proof of Lemma \ref{identical maps}] By abuse of notation, let $\tilde \mu_{n+1}$ and $\tilde \mu_{n}$ be the measures $\tilde \mu_{n+1}$ and $\tilde \mu_{n}$ in \eqref{measure tilde i} restricted on $D_n(\tilde a, r_1)$ and respectively, on $D_n(\tilde b, r_2)$ for $\tilde b=(a_1, \dots, a_{n-1}, b_0)$ and  small radii $r_1, r_2$.  Since $\hat \mu_n=\hat \mu_{n+1}$, there exist constants $c_\ell>0$, such that 
   $$\Phi_\ell^*(\tilde \mu_{n+1})=c_\ell \cdot \tilde \mu_{n+1}.$$
By Proposition \ref{equal currents}, we see that for any $\underline \alpha$ in $D_{n-1}(\underline a, r_1)\cap J_{f_1}\times \cdots \times J_{f_{n-1}}$, the map $\hbar_\ell(\underline \alpha, \cdot)$ is a symmetry of $J_{f_n}$ on $D(a_n, r_2)$. Moreover, the functions $\hbar_\ell(\tilde x)$ converge uniformly to $\hbar(\tilde x):=x_n$ on $D_n(\tilde a, r_1)$ as $\ell$ tends to infinity. Applying Lemma \ref{identical symmetry}, there exists $N\in \N$, such that for any $\ell\geq N$ and any $\underline \alpha$ in $D_{n-1}(\underline a, r_1)\cap \left(J_{f_1}\times \cdots \times J_{f_{n-1}}\right)$, we have 
   $$\hbar_\ell(\underline \alpha, x_n)= x_n$$
for each $x_n\in D(a_n, r_1)$. Since $a_i$ is an accumulating point in $J_{f_i}$ for each $i$ (see \cite{Milnor:book}), when $\ell\geq N$, the zero locus of the equation $\hbar_\ell(\tilde x )-x_n=0$ on $D_n(\tilde a, r_1)$ cannot have dimension $\leq n-1$, i.e., $\hbar_\ell(\tilde x )$ is identically equal to $x_n$ and so, $\Phi_\ell$ is the identity map. This concludes the proof of Lemma~\ref{identical maps}. 
\end{proof}

Let $N$ be the positive integer appearing in Lemma \ref{identical maps}. Pick $\ell_2>\ell_1>N$ with $j_{\ell_2}\ge j_{\ell_1}$ and $j_{i, \ell_2}\ge j_{i, \ell_1}$ for $1\leq i\leq n$. Let 
  $$m_i:=j_{i,\ell_2}-j_{i,\ell_1}\textup{ for $1\le i\le n$ and } m_{n+1}:=j_{\ell_2}-j_{\ell_1}.$$

\begin{lemma}
\label{first invariant}
With the above notation for the $m_i$'s, let 
$$H':=(f_1^{m_1}, \cdots, f_{n+1}^{m_{n+1}})(H)\subset (\bP^1)^{n+1}_\C.$$
Then $(f_1^{m_1}, \cdots, f_{n+1}^{m_{n+1}})(H')=H'$. 
\end{lemma}

\begin{proof}[Proof of Lemma~\ref{first invariant}.]
From Lemma \ref{identical maps} (see also \eqref{split equation}), we have that  
$$g^{j_{\ell_1}}\circ h\circ (f_1^{-j_{1,{\ell_1}}}, \cdots, f_n^{-j_{n,{\ell_1}}})(\tilde x)=g^{j_{\ell_2}}\circ h\circ (f_1^{-j_{1,{\ell_2}}}, \cdots, f_n^{-j_{n,{\ell_2}}})(\tilde x)$$
on $D_n(\tilde a, r_0)$, or equivalently 
\begin{equation}\label{identical hypers}
h(\tilde x)=g^{m_{n+1}}\circ h\circ (f_1^{-m_1}, \cdots, f_n^{-m_n})(\tilde x).
\end{equation}
Let 
  $$h'(\tilde x): =f_{n+1}^{m_{n+1}}\circ h\circ (f_1^{-m_1}, \cdots, f_n^{-m_n})(\tilde x)$$
on a neighborhood of $\tilde a$. Now consider the analytic equation
   $$h'(\tilde x)-x_{n+1}=0$$
 on a neighbourhood of $(\tilde a, b_1)\in \mathbb{P}_\C^n\times \mathbb{P}_\C^1$. The zero set of this equation is an analytic set of dimension $n$ passing through the point $(\tilde a, b_1)$. For $\tilde x$ close to $\tilde a$, the points of the form $(\tilde x, h'(\tilde x))$ lie on the hypersurface $H'$.
Combining \eqref{def of g} and \eqref{identical hypers}, we get  
  $$ h'\circ  (f_1^{m_1}, \cdots, f_n^{m_n})(\tilde x)=f^{m_{n+1}}\circ h'(\tilde x).$$
Hence for points $\tilde x$ close to $\tilde a$, the points $(f_1^{m_1}, \cdots, f_{n+1}^{m_{n+1}})(\tilde x, h'(\tilde x))$, which are points on $(f_1^{m_1}, \cdots, f_{n+1}^{m_{n+1}})(H')$  satisfy also the equation $h'(\tilde x)-x_{n+1}=0$. Finally, as both $H'$ and  $(f_1^{m_1}, \cdots, f_{n+1}^{m_{n+1}})(H')$ share an analytic set of dimension $n$ in a neighbourhood of $(\tilde a, b_1)$, they must be identical. So $H'$ is fixed by the endomorphism $(f_1^{m_1}, \cdots, f_{n+1}^{m_{n+1}})$ of $(\bP^1)^{n+1}$, as desired. 
\end{proof}

We recall that our sequence of tuples $(j_\ell,j_{1,\ell},\dots, j_{n,\ell})$ satisfies condition \eqref{arithmetic condition}. Therefore, we can choose some integers $\ell_2>\ell_1>N$ such that $j_{\ell_2}=j_{\ell_1}$ and also, $j_{i,\ell_2}=j_{i,\ell_1}$ for $i=2,\dots, n$, while $j_{1,\ell_2}=j_{1,\ell_1}+1$ and then apply Lemma~\ref{first invariant} to the tuple of integers 
$$m_i:=j_{i,\ell_2}-j_{i,\ell_1}\text{ for $1\le i\le n$ and }m_{n+1}:=j_{\ell_2}-j_{\ell_1}.$$
We have that $m_{i}=0$ for each $i=2,\dots, n+1$, while $m_1=1$.  
Therefore, Lemma~\ref{first invariant} yields that 
\begin{equation}
\label{eq third invariant}
\left(f_1^2,{\rm id},\cdots, {\rm id}\right)(H)=\left(f_1,{\rm id},\cdots, {\rm id}\right)(H),
\end{equation}
where the action in \eqref{eq third invariant} on coordinates $x_i$ for $2\le i\le n+1$ is given by the corresponding identity maps. 
Equation \eqref{eq third invariant} yields that $H$ is a hypersurface of the form $\bP^1\times H_0$ (for some hypersurface $H_0\subset (\bP^1)^n$), contradicting thus our hypothesis that $H$ projects dominantly onto any subset of $n$ coordinate axes. Hence $f_{n+1}$ (and thus each of the $f_i$'s, as shown in Corollary~\ref{non-exceptional}) must be exceptional; this concludes our proof of Theorem~\ref{measure to periodicity}.
\end{proof}


\section{Conclusion of our proof}
\label{section proof of main results}

In this Section we finish our proof of Theorem~\ref{general hypersurface theorem} and then we prove Theorem~\ref{shape preperiodic}. Since we showed in Proposition~\ref{prop first reduction} that it suffices to assume in Theorems~\ref{main result}~and~\ref{general DMM result} that the subvariety $V\subset (\bP^1)^n$ is a hypersurface projecting dominantly onto each subset of $(n-1)$ coordinate axes, then this will conclude our proof for both of those two theorems.

\begin{proof}[Proof of Theorem~\ref{general hypersurface theorem}.]
So, we have a hypersurface $H\subset (\bP^1)^n$ (for some integer $n>2$)  containing a Zariski dense set of  points satisfying either hypothesis~(1) or hypothesis~(2) in Theorem~\ref{general hypersurface theorem}. Furthermore, $H$ projects dominantly onto any subset of $(n-1)$ coordinate axes of $(\bP^1)^n$. We let $\hat \mu_i$ (for $i=1,\dots, n$) be the measures introduced in Subsection~\ref{subsection measures 2}.

\begin{lemma}
\label{equal measures lemma}
We have $\hat \mu_1=\hat \mu_2=\cdots = \hat \mu_n$. 
\end{lemma}

\begin{proof}[Proof of Lemma~\ref{equal measures lemma}.]
If each $f_i$ and also $H$ are defined over $\Qbar$ (i.e, hypothesis~(2) in Theorem~\ref{general hypersurface theorem} is met), then the conclusion of Lemma~\ref{equal measures lemma} follows immediately from Theorem~\ref{equal measure thm}. So, assume now that each $f_i$ and also $H$ are defined over $\C$, and moreover hypothesis~(1) in Theorem~\ref{general hypersurface theorem} is met; in particular, $\deg(f_1)=\deg(f_2)=\cdots = \deg(f_n)$. We prove the result in this general case using a specialization technique similar to the one employed in the proof of Claim~\ref{specialization claim}. 

So, we let $K$ be a finitely generated subfield of $\C$ such that each $f_i$ and also $H$ are defined over $K$, and let $\Kbar$ be a fixed algebraic closure of $K$ in $\C$. We know there exists an infinite sequence $S:=\{(x_{1,j}, \dots, x_{n,j})\}\subset H(\C)$ such that each $x_{i,j}$ is a preperiodic point for $f_i$ for $i=1,\dots, n$ and for each $j\ge 1$. Then the $f_i$'s are  base changes of endomorphisms $f_{i,K}$  of $\bP^1_K$ (for $i=1,\dots, n$); similarly, $S$ is the base change of a subset $S_K\subset H(\Kbar)$.  We can further extend $f_{i,K}$ to endomorphisms
$$f_{i,U}:  \bP^1_U\lra \bP^1_U$$
over a variety $U$ over $\Q$ of finite type and with function field $K$. 
For each geometric point $t\in U(\Qbar)$, the objects $f_{i,U}$ and $S_U$ have reductions $f_{i,t}$ and $S_t$ such that $S_t$ consists of points with  coordinates preperiodic under the action of the $f_{i,U}$'s. We also let $\hat{\mu}_{i,t}$ (for $i=1,\dots, n$) be the probability measures on $H_t$ obtained as pullback through the usual projection map onto $(n-1)$ coordinates (with the exception of the $i$-th coordinate axis) of the invariant measures on $(\bP^1_\C)^{n-1}$ corresponding to each $f_{j,t}$ for $j\ne i$.  
As proven in Claim~\ref{specialization claim} (using \cite[Theorem~4.7]{Yuan-Zhang-2} and also \cite[Lemma~3.2.3]{Yuan-Zhang-1}), we obtain that the subset $S_t\subset H_t$ is still Zariski dense for all the $\Qbar$-points $t$ in a dense open subset $U_0\subseteq U$. Thus, as proven in Theorem~\ref{equal measure thm}, we conclude that 
$$\hat{\mu}_{1,t}=\hat{\mu}_{2,t}=\cdots =\hat{\mu}_{n,t}$$ 
for each $t\in U_0(\Qbar)$. Since $U_0(\Qbar)$ is dense in $U(\C)$ with respect to the usual archimedean topology, while the measures $\hat{\mu}_{i,t}$ vary continuously with the parameter $t$ (since from the construction,  the potential functions of these measures vary continuously with the coefficients of $f_{i,t}$), we conclude that 
$$\hat{\mu}_{1,t}=\hat{\mu}_{2,t}=\cdots =\hat{\mu}_{n,t}$$ 
for all points in $U(\bC)$ including the point corresponding to the original embedding $K\subset \C$. Thus $\hat \mu_1=\hat \mu_2=\cdots =\hat \mu_n$, which   concludes the proof of Lemma~\ref{equal measures lemma}. 
\end{proof}

Lemma~\ref{equal measures lemma} yields that the hypotheses of Proposition~\ref{iff proposition} are met and so, we know that there exists an index $i$, which we assume (without loss of generality) to be $n$ so that for each $\alpha:=(a_1,\dots, a_n)\in H(\C)$, if $a_i$ is preperiodic under the action of $f_i$ for $i=1,\dots, n-1$, then also $a_n$ is preperiodic under the action of $f_n$. 

Since all but finitely many periodic points of a rational map are repelling, and also, there is a Zariski dense open subset of points $\alpha\in H$ such that the restriction of the natural projection map $\pi|_H:H\lra (\bP^1)^{n-1}$ on the first $(n-1)$ coordinate axes is unramified, then we can find a point $(x_{1,0},\dots, x_{n,0})\in H(\C)$ satisfying the following properties:
\begin{itemize}
\item[(a)] $x_{i,0}$ is a periodic repelling point for $f_i$ for each $i=1,\dots, n-1$; and
\item[(b)] there is a non-constant holomorphic germ $h_0$ defined in a neighborhood of $\tilde{x}_0:=(x_{1,0},\dots, x_{n-1,0})$, with $h_0\left(\tilde{x}_0\right)=x_{n,0}$ and $\left(\tilde{x},h_0(\tilde{x})\right)\in H(\C)$ for all $\tilde{x}$ in a small neighborhood of $\tilde{x}_0$. Moreover, we also have that 
\begin{equation}
\label{closed condition 0}
\frac{\partial h_0}{\partial x_i}(\tilde{x}_0)\neq 0\text{ for each $i=1,\dots, n-1$.}
\end{equation}
\end{itemize}
Note that hypothesis \eqref{closed condition 0} can be achieved since the points satisfying $\frac{\partial h}{\partial x_i}= 0$ live in a proper Zariski closed subset of $H$ (i.e., inequality \eqref{closed condition} is an \emph{open condition} which can be seen from computing the partial derivatives using implicit functions). It is essential in this case to know that $H$ projects dominantly onto each subset of $(n-1)$ coordinates, i.e., $H$ is \emph{not} of the form $\bP^1\times H_0$ for some hypersurface $H_0\subset (\bP^1)^{n-1}$ since otherwise condition~\ref{closed condition 0} would not necessarily hold.

Proposition~\ref{iff proposition} and condition~(a) above yield that $x_{n,0}$ is preperiodic for $f_n$. At the expense of replacing each $f_i$ by $f_i^\ell$ (for a suitable positive integer $\ell$), we may assume that
\begin{itemize}
\item $x_{i,0}$ is a repelling fixed  point of $f_i$ for $1\leq i\leq n-1$;  
\item $x_{n,1}:=f_{n}(x_{n,0})$ is a fixed point of $f_{n}$; and
\item there is a holomorphic germ $h(\tilde x)$ near $\tilde{x}_0=(x_{1,0},\dots, x_{n-1,0})$ with $h(\tilde{x}_0)=x_{n,0}$,  and $(\tilde x, h(\tilde x))\in H(\C)$ for all $\tilde x\in (\bP^1)^{n-1}(\C)$ in a small (complex analytic) neighbourhood of $\tilde{x}_0$. Moreover, for each $i=1, \cdots, n-1$ we have that 
\begin{equation}
\label{closed condition}
\beta_i:=\frac{\partial h}{\partial x_i}(\tilde{x}_0)\neq 0.
\end{equation}
\end{itemize}
Then all hypotheses in Theorem~\ref{measure to periodicity} are met; this yields that each $f_i$ must be either all conjugate to monomials and $\pm$Chebyshev polynomials, or they are all Latt\'es maps, which concludes our proof of Theorem~\ref{general hypersurface theorem}.   
\end{proof}

We finish our paper by proving Theorem~\ref{shape preperiodic}.

\begin{proof}[Proof of Theorem~\ref{shape preperiodic}.]
First we observe (similar to the proof of Proposition~\ref{prop first reduction}) that it suffices to prove that each irreducible, preperiodic  hypersurface  $H\subset (\bP^1)^n$ is of the form $\pi_{i,j}^{-1}\left(C_{i,j}\right)$ (for a pair of indices $i,j\in\{1,\dots, n\}$), where $C_{i,j}\subset \bP^1\times \bP^1$ is a curve, which is preperiodic under the action of $(x_i,x_j)\mapsto \left(f_i(x_i), f_j(x_j)\right)$ (and $\pi_{i,j}$ is the projection of $(\bP^1)^n$ onto the $(i,j)$-th coordinate axes). Indeed, just as in the proof of Proposition~\ref{prop first reduction}, we obtain that any preperiodic subvariety $V\subset (\bP^1)^n$ is a component of an intersection of preperiodic hypersurfaces, thus reducing our proof to the case $V$ is a hypersurface. 

Since the case $n=2$ was proved in \cite[Theorem~1.1]{GNY}, from now on, we assume $V\subset (\bP^1)^n$ is a hypersurface and $n>2$. Then, at the expense of replacing $\Phi=(f_1,\dots, f_n)$ by an iterate of it and also replacing the hypersurface $V$ by a suitable $\Phi^k(V)$ (for $k\in\N$), we may (and do) assume that $V$ is invariant under the action of $\Phi$. Also, we may assume $V$ projects dominantly onto each subset of $(n-1)$ coordinate axes of $(\bP^1)^n$ since otherwise $V=\bP^1\times V_0$ and then we can argue inductively on $n$ (because $V_0\subset (\bP^1)^{n-1}$ would be invariant under the induced action of $\Phi$ on those $(n-1)$ coordinate axes). Next we will prove there are \emph{no} such hypersurfaces, thus providing the desired conclusion in Theorem~\ref{shape preperiodic}.

We let $\pi|_V:V\lra (\bP^1)^{n-1}$ be the projection on the first $n-1$ coordinate axes; we know there exists a Zariski open subset $U\subset (\bP^1)^{n-1}$ such that $\pi|_V^{-1}(\beta)$ is finite for each $\beta\in U$. 

Now, let $\beta:=(a_1,\dots, a_{n-1})\in U(\C)$ such that each $a_i$ is periodic under the action of $f_i$. We claim that each point $\alpha\in V(\C)$ satisfying  $\pi|_V(\alpha)=\beta$ is preperiodic under the action of $\Phi$, i.e., its last coordinate is preperiodic for $f_n$. Indeed, since $\beta$ is periodic, then for some positive integer $m$, we have that $\Phi^m(\alpha)\in \pi|_V^{-1}(\beta)$ and because $\pi|_V^{-1}(\beta)$ is a finite set, we conclude that the last coordinate of $\alpha$ (and therefore, $\alpha$ itself) must be preperiodic, as claimed.

At the expense of shrinking $U$ to a smaller, but still Zariski dense, open subset, we may even assume $\pi|_V$ is unramified above each point of $U$. Then we can argue as in the proof of Theorem~\ref{general hypersurface theorem} and find a point $(x_{1,0},\dots, x_{n,0})$ satisfying the conditions: 
\begin{itemize}
\item[(a)] $x_{i,0}$ is a periodic repelling point for $f_i$ for each $i=1,\dots, n-1$; and
\item[(b)] there is a non-constant holomorphic germ $h_0$ defined in a neighborhood of $\tilde{x}_0:=(x_{1,0},\dots, x_{n-1,0})$, with $h_0\left(\tilde{x}_0\right)=x_{n,0}$ and $\left(\tilde{x},h_0(\tilde{x})\right)\in V(\C)$ for all $\tilde{x}$ in a small neighborhood of $\tilde{x}_0$. Moreover, we also have that $\frac{\partial h_0}{\partial x_i}(\tilde{x}_0)\neq 0$ for each $i$. 
\end{itemize}
Furthermore, after replacing $\Phi$ by yet another iterate, we get that each $x_{i,0}$ is fixed by $f_i$. Then we meet the hypotheses of Theorem~\ref{measure to periodicity} and since we assumed that each $f_i$ is non-exceptional, we derive a contradiction. This concludes our proof of Theorem~\ref{shape preperiodic}. 
\end{proof}





\begin{thebibliography}{AKLS09}
\newcommand{\au}[1]{{#1},}
\newcommand{\ti}[1]{\textit{#1},}
\newcommand{\jo}[1]{{#1}}
\newcommand{\vo}[1]{\textbf{#1}}
\newcommand{\no}[1]{no. {#1},}
\newcommand{\yr}[1]{(#1),}
\newcommand{\pp}[1]{#1.}
\newcommand{\ppx}[1]{#1,}
\newcommand{\pps}[1]{#1;}
\newcommand{\bk}[1]{{#1},}
\newcommand{\inbk}[1]{in: {#1}}
\newcommand{\xxx}[1]{{arXiv:#1}}



\bibitem[Bak09]{Baker}
\au{M.~Baker}
\ti{A finiteness theorem for canonical heights attached to rational maps over function fields} 
\jo{J. Reine Angew. Math.}
\vo{626}
\yr{2009}
\pp{205-233}

\bibitem[BH05]{Baker-Hsia}
\au{M.~Baker and L.-C.~Hsia}
\ti{Canonical heights, transfinite diameters, and polynomial dynamics}
\jo{J. reine angew. Math.}
\vo{585}
\yr{2005}
\pp{61--92}


\bibitem[BR06]{BR}
M.~Baker and R.~Rumely, \emph{Equidistribution of small points, rational
  dynamics, and potential theory}, Ann. Inst. Fourier (Grenoble) \textbf{56}
  (2006), 625--688.

\bibitem[Ben05]{Rob}
\au{R.~L.~Benedetto}
\ti{Heights and preperiodic points of polynomials over function fields}
\jo{Int. Math. Res. Not.}
\vo{62} 
\yr{2005}
\pp{3855-3866} 


\bibitem[BG06]{Bombieri}
\au{E.~Bombieri and W.~Gubler}
\ti{Heights in Diophantine geometry} 
New Mathematical Monographs \textbf{4}, Cambridge University Press, Cambridge, 2006. xvi+652 pp.

\bibitem[Bro65]{Bro}
H. ~Brolin, \emph{Invariant sets under iteration of rational functions},  Ark. Mat. {\bf 6}(1965), 103--144.




\bibitem[CS93]{C-S}
G. ~Call and J.~Silverman, \emph{Canonical heights on varieties with morphisms},  Compositio Math. \textbf{89}(1993), 163--205.

\bibitem[CL06]{C-L}
A.~Chambert-Loir,
\newblock{\em Mesures et \'equidistribution sur les espaces de Berkovich},
\newblock{ J. Reine Angew. Math.} {\bf 595} (2006), 215--235.

\bibitem[DF]{Favre-Dujardin}
R.~Dujardin and C.~Favre, \emph{The dynamical Manin-Mumford problem for plane polynomial automorphisms},  J. Eur. Math. Soc., 2016, to appear.

\bibitem[DH93]{DH}
A. ~Douady and J. ~Hubbard, \emph{A proof of Thurston's topological characterization of rational functions}, Acta. Math. {\bf 171} (1993) 263--297.

\bibitem[Fab09]{Faber}
\au{X.~W.~C.~Faber}
\ti{Equidistribution of dynamically small subvarieties over the function field of a curve}
\jo{Acta Arith.}
\vo{137}
\yr{2009} 
\pp{345--389}



\bibitem[Fal84]{Faltings-Hodge}
G.~Faltings, \emph{Calculus on arithmetic surfaces}, Ann. Math. \textbf{119} (1984), 387-424.


\bibitem[FRL06]{favre-rivera06}
C.~Favre and J.~Rivera-Letelier, 
  \emph{\'Equidistribution quantitative des points de petite hauteur sur
    la droite projective}, Math. Ann. \textbf{355} (2006), 311--361.


\bibitem[FLM83]{FLM}
A. ~Freire, A. ~Lopes and R. ~Ma\~n\'e, \emph{An invariant measure for rational maps},  Bol. Soc. Brasil
Mat. \textbf{14} (1983) 45-62.


\bibitem[GN16]{DBHC}
\au{D.~Ghioca and K.~D.~Nguyen}
\ti{Dynamical Anomalous Subvarieties: Structure and Bounded Height Theorems} \jo{Adv. Math.} 
\vo{288}
\yr{2016} 
\pp{1433--1462}

\bibitem[GN]{DPZ}
\au{D.~Ghioca and K.~D.~Nguyen}
\ti{The Dynamical Pink-Zilber Conjecture} 
preprint.

\bibitem[GNY]{GNY}
\au{D.~Ghioca, K.~D.~Nguyen and H.~Ye}
\ti{The Dynamical Manin-Mumford Conjecture and the Dynamical Bogomolov Conjecture for split rational maps}
\jo{J. Eur. Math. Soc.,}
2017, 25 pages, to appear.

\bibitem[GT10]{GT-Bogomolov}
\au{D.~Ghioca and T.~J.~Tucker} 
\ti{Proof of a dynamical Bogomolov conjecture for lines under polynomial actions}
\jo{Proc. Amer. Math. Soc.}
\vo{138}
\yr{2010}
\pp{937--942}

\bibitem[GTZ11]{GTZ}
\au{D.~Ghioca, T.~J.~Tucker and S.~Zhang}
\ti{Towards a dynamical Manin-Mumford Conjecture}
\jo{Int. Math. Res. Not. IMRN}
2011, no. 22, 5109--5122.



\bibitem[Gub08]{Gubler} 
\au{W.~Gubler}
\ti{Equidistribution over function fields} 
\jo{Manuscripta Math.}
\vo{127}
\yr{2008} 
\pp{485-510}

\bibitem[Ham95]{Ham}
D. ~Hamilton, \emph{Length of Julia curves}, Pacific J. Math. {\textbf{169}} (1995), no. 1, 75--93.

\bibitem[Hri85]{Hriljac}
P.~Hriljac, \emph{Heights and Arakelovs intersection theory}, Amer. J. Math. \textbf{107} (1985), 23-38.

\bibitem[Ino11]{Inou}
H.~Inou, \emph{Extending local analytic conjugacies}, Trans. Amer. Math. Soc. \textbf{363} (2011), no.~1, 331--343.


\bibitem[Lau84]{Laurent}
M. Laurent, \emph{Equations diophantiennes exponentielles}, Invent. Math. \textbf{78} (1984), 299--327.

\bibitem[Lev90]{Levin}
G.~M. Levin, \emph{Symmetries on {J}ulia sets}, Mat. Zametki \textbf{48}
  (1990), 72--79, 159.

\bibitem[Lyu83]{Lyu1}
M. ~Lyubich, \emph{Entropy properties of rational endomorphisms of the Riemann sphere},  Erg. Th.
and Dyn. Sys. \textbf{3} (1983), 351--385.



\bibitem[Man83]{Man}
R. ~Ma\~n\'e, \emph{On the uniqueness of the maximizing measure for rational maps}, Bol. Soc. Bras.
Math. {\bf 14.2} (1983), 27--43.




\bibitem[McQ95]{McQuillan}
M.~McQuillan, \emph{Division points on semi-abelian varieties}, 
Invent. Math. \textbf{120} (1995), no. 1, 143--159. 

\bibitem[Med07]{Alice}
A.~Medvedev, \emph{Minimal sets in ACFA}, Thesis (Ph.D.)--University of California, Berkeley, 2007, 96 pp.

\bibitem[MS14]{Medvedev-Scanlon}
A. Medvedev and T. Scanlon, \emph{Invariant varieties for polynomial dynamical systems},
Ann. of Math. (2) \textbf{179} (2014), no. 1, 81--177.


\bibitem[Mil00]{Milnor:book}
J.~Milnor, \emph{Dynamics in one complex variable}, Wiesbaden, Germany: Vieweg. ISBN 3-528-13130-6. (2000).


\bibitem[Mil04]{lattes}
J.~Milnor, \emph{On Latt\`es maps}, preprint available online at  http://arxiv.org/abs/math/0402147.

\bibitem[Mim97]{Mimar-thesis}
A.~Mimar, \emph{On the preperiodic points of an endomorphism of $\bP^1\times \bP^1$  which lie on a curve}, Thesis (Ph.D.), Columbia University, 1997.


\bibitem[Mim13]{Mimar}
A.~Mimar, \emph{On the preperiodic points of an endomorphism of $\bP^1\times \bP^1$ which lie on a curve}, Trans. Amer. Math. Soc. \textbf{365} (2013), no.~1, 161--193.



\bibitem[Mor96]{Moriwaki}
A.~Moriwaki, \emph{Hodge index theorem for arithmetic cycles of codimension one}, Math. Res. Lett. \textbf{3} (1996), no.~2, 173-183.





\bibitem[Ray83]{Raynaud}
M.~Raynaud, \emph{Sous-vari\'{e}t\'{e}s d'une vari\'{e}t\'{e} ab\'{e}lienne et
  points de torsion}, Arithmetic and geometry, vol. I, Progr. Math., vol.~35,
  Birkh\"{a}user, Boston, MA, 1983, pp.~327--352.



\bibitem[Thu85]{Thu}
W.~Thurston, \emph{On the combinatorics of iterated rational maps}, preprint, 1985.

\bibitem[Ull98]{Ullmo}
E.~Ullmo, \emph{Positivit\'e et discr\'etion des points alg\'ebriques des
  courbes}, Ann. of Math. (2) \textbf{147} (1998), 167--179.


\bibitem[Yua08]{Yuan}
X.~Yuan, \emph{Big line bundles over arithmetic varieties},
 Invent. Math. \textbf{173} (2008), 603--649.

\bibitem[YZa]{Yuan-Zhang-1}
X. Yuan and S. Zhang, \emph{Calabi theorem and algebraic dynamics}, preprint, 24 pages. 

\bibitem[YZb]{Yuan-Zhang-2}
X.~Yuan and S.~Zhang, \emph{Small points and Berkovich metrics}, 28 pages, preprint available on the webpage: {\tt http://www.math.columbia.edu/\~{ }szhang/papers/Preprints.html}

\bibitem[YZc]{Yuan-Zhang-arxiv}
X.~Yuan and S.~Zhang, \emph{The arithmetic Hodge index theorem for adelic line bundles II},  arXiv:1304.3539. 

\bibitem[YZ16]{Yuan-Zhang-published}
X.~Yuan and S.~Zhang, \emph{The arithmetic Hodge index theorem for adelic line bundles}, Math. Ann. \textbf{367} (2017), no.~3, 1123--1171.


\bibitem[Zan12]{Zannier}
U.~Zannier, \emph{Some problems of unlikely intersections in arithmetic and
  geometry}, Annals of Mathematics Studies, vol. 181, Princeton University
 Press, Princeton, NJ, 2012, With appendixes by David Masser.
 

\bibitem[Zdu90]{Zdu}
A. Zdunik, \emph{Parabolic orbifolds and the dimension of the maximal measure for rational maps}, Inv. Math. {\bf 99} (1990) 627--649.

\bibitem[Zha95a]{Zhang:line}
S. Zhang, \emph{Positive line bundles on arithmetic varieties},
 J. Amer. Math. Soc. \textbf{8} (1995), 187--221. 

\bibitem[Zha95b]{Zhang:metrics}
S.~Zhang, \emph{Small points and adelic metrics},
 J. Alg. Geom. \textbf{4} (1995), 281--300. 

\bibitem[Zha98]{Zhang:Bogomolov}
S.~Zhang, \emph{Equidistribution of small points on abelian varieties}, Ann. Math. (2), {\bf 147}(1998), 159--165.

\bibitem[Zha06]{ZhangLec}
S.~Zhang, \emph{Distributions in algebraic dynamics}, Survey in
  Differential Geometry, vol.~10, International Press, 2006, pp.~381--430.

\end{thebibliography}
\end{document}